\documentclass[11pt]{amsart}
\usepackage{amsmath, amssymb}
\usepackage{amsfonts}
\usepackage{mathrsfs}
\usepackage[arrow,matrix,curve,cmtip,ps]{xy}
\usepackage{epsfig}
\usepackage{amsthm}
\usepackage{tikz-cd}

\usepackage[numbers]{natbib}

\allowdisplaybreaks

\newtheorem{theorem}{Theorem}[section]
\newtheorem{lemma}[theorem]{Lemma}
\newtheorem{proposition}[theorem]{Proposition}
\newtheorem{corollary}[theorem]{Corollary}

\newtheorem*{theorem*}{Theorem}
\theoremstyle{remark}
\newtheorem{remark}[theorem]{Remark}
\newtheorem{definition}[theorem]{Definition}
\newtheorem{example}[theorem]{Example}

\newcommand{\N}{\mathbb{N}}

\newcommand{\C}{\mathbb{C}}

\newcommand{\vp}{\varphi}
\newcommand{\norm}[1]{\left\lVert #1 \right\rVert}
\newcommand{\wt}[1]{\widetilde{#1}}
\newcommand{\id}{\text{id}}

\newcommand{\tx}{\text}
\newcommand{\cal}[1]{\mathcal{#1}}
\newcommand{\Span}{\text{span}}

\DeclareMathAlphabet\mathbfcal{OMS}{cmsy}{b}{n}

\begin{document}

\title[$k$-AOU spaces and quantum correlations]{Matricial Archimedean order unit spaces and quantum correlations}

\author{Roy Araiza}
\author{Travis Russell}
\author{Mark Tomforde}

\address{Department of Mathematics \& IQUIST, University of Illinois at Urbana-Champaign, Urbana, IL 61801}
\email{raraiza@illinois.edu}
\address{Army Cyber Institute, United States Military Academy, West Point, NY 10996}
\email{travis.russell@westpoint.edu}
\address{Department of Mathematics \\ University of Colorado \\ Colorado Springs, CO 80918 \\USA}
\email{mtomford@uccs.edu}

\thanks{The first author was partially supported as an Andrews Fellow by the Department of Mathematics at Purdue University and as a J.L. Doob Research Assistant Professor in the Department of Mathematics at the University of Illinois at Urbana-Champaign. The third author was supported by a grant from the Simons Foundation (\#527708 to Mark Tomforde)}

\date{\today}



\begin{abstract}
We introduce a matricial analogue of an Archimedean order unit space, which we call a $k$-AOU space. We develop the category of $k$-AOU spaces and $k$-positive maps and exhibit functors from this category to the category of operator systems and completely positive maps. We also demonstrate the existence of injective envelopes and C*-envelopes in the category of $k$-AOU spaces. Finally, we show that finite-dimensional quantum correlations can be characterized in terms of states on finite-dimensional $k$-AOU spaces. Combined with previous work, this yields a reformulation of Tsirelson's conjecture in terms of operator systems and $k$-AOU spaces.
\end{abstract}

\maketitle

\section{Introduction}

The origin of the theory of \emph{Archimedean order unit spaces} dates back to the work of Kadison \cite{kadison1951representation} in the mid 20\textsuperscript{th} century. Its development provided the foundations for the study of \emph{operator systems}, which have had a multitude of applications to many areas in operator algebras as well as the rapidly growing field of quantum information theory. Every Archimedean order unit space may be equipped with an operator system structure, and in particular, such a space induces a canonical minimal and maximal operator system structure \cite{PaulsenTodorovTomfordeOpSysStructures}. 

It was shown by Xhabli in \cite{xhabli2012super} that in addition to  these minimal and maximal structures,  one may consider ``intermediate'' structures on an operator system. Xhabli proved that if $\mathcal V$ is an operator system, then for any natural number $k \in \mathbb N$ one may consider the \emph{super k-minimal} operator system $\text{OMIN}_k(\mathcal V)$ and the \emph{super k-maximal} operator system $\text{OMAX}_k(\mathcal V).$  Xhabli showed that the identity map $\id: \mathcal V \to \text{OMIN}_k(\mathcal V)$ (respectively, the identity map $\id: \text{OMAX}_k(\mathcal V) \to \mathcal V$) is a unital $k$-order isomorphism and is completely positive. Thus the operator systems coincide up to the $k$\textsuperscript{th}-level.

Given a natural number $k$, we consider \emph{k-Archimedean order unit spaces} and develop their theory. A $k$-Archimedean order unit space consists of a $*$-vector space $\mathcal V$, a cone $C \subseteq M_k(\mathcal V)$, and a unit $e \in \mathcal V$, such that the triple $(M_k(\mathcal V), C, I_k \otimes e)$ an Archimedean order unit space and the cone is invariant under conjugation by scalar matrices (see Definition~\ref{defn: k-aou space}). We consider operator system structures on a $k$-Archimedean order unit space $\mathcal{V}$ for which the positive elements of $M_k(\mathcal{V})$ coincide with the cone $C$. Given a $k$-Archimedean order unit space $\mathcal{V}$, we define a \emph{$k$-minimal operator system} $\mathcal V_{k\text{-min}}$ and a \emph{k-maximal} operator systems $\mathcal V_{k\text{-max}}$, which satisfy the same universal properties as their super operator system counterparts (see Proposition~\ref{prop: properties of V_k-min} and Proposition~\ref{prop: k-positive implies cp on k-min}). We show that the vital information of these operator systems is contained in the structure at the $k$\textsuperscript{th} matrix level. We then proceed to develop the category whose objects are $k$-Archimedean order unit spaces and whose morphisms are unital $k$-positive maps. In analogy with the pioneering work of Hamana \cite{hamana1979injective}, we prove in Theorem~\ref{thm: existence of injective envelope in k-AOU} that every $k$-Archimedean order unit space has a unique injective envelope in this category. Moreover, we prove in Theorem~\ref{thm: the injective envelope of k-min is also k-min} that the injective envelope of a $k$-Archimedean order unit space $\mathcal{V}$ coincides with the operator system injective envelope of $\mathcal{V}_\text{$k$-min}$. We also show in Corollary~\ref{cor: C^*_e(V) k-minimal when V is k-minimal} that every $k$-Archimedean order unit space has a unique C*-envelope that coincides with the C*-envelope of $\mathcal{V}_\text{$k$-min}$, in the sense that the C*-envelope of $\mathcal{V}_\text{$k$-min}$ is completely order isomorphic to $\mathcal{W}_\text{$k$-min}$ for some $k$-Archimedean order unit space $\mathcal{W}$ containing $\mathcal{V}$.

Our investigation into $k$-Archimedean order unit spaces is motivated by recent progress in the theory of quantum correlations, and as a consequence we are able to provide applications to this theory. Given two natural numbers $m$ and $n$, the \emph{nonsignalling correlations $C_{ns}(n,m)$} with $n$-inputs and $m$-outputs are tuples in $\mathbb R^{m^2 n^2}$ consisting of positive real numbers that induce well-defined \emph{marginal densities}. Such correlations model bipartite systems where two parties receive inputs on which they make measurements to produce various outputs (see Section~\ref{sec: quantum correlations and k-AOU spaces} for details). Of special interest are the convex subsets of \emph{quantum commuting} correlations, denoted $C_{qc}(n,m)$, and the \emph{quantum} correlations, denoted $C_q(n,m)$. While the relation $C_q(n,m) \subseteq C_{qc}(n,m)$ follows from their definitions, the question of whether or not this containment is proper (or, more specifically, whether the closure of $C_q(n,m)$ is equal to $C_{qc}(n,m)$) has generated tremendous activity since it was posed by Tsirelson in the 1980s \cite{cirel1980quantum}. Through the work of various authors, Tsirelson's question was proven to be equivalent to Connes' Embedding Problem, a major problem in von Neumann algebra theory.  (One direction of the equivalence was proven independently by Fritz \cite{fritz2012tsirelson} and by Junge, Navascues, Palazuelos, Perez-Garcia, Scholz, and Werner  \cite{junge2011connes}, and the other direction of the equivalence was established by Ozawa \cite{ozawa2013connes}.)  By the work of Kirchberg in \cite{kirchberg1993non}, Connes' Embedding Problem developed into one of the most famous open problems in the field of operator algebras. It was recently announced that the closure of $C_q(n,m)$ is a proper subset of $C_{qc}(n,m)$ for some values of $m$ and $n$ \cite{ji2020mip}, thus settling the questions of Tsirelson, Kirchberg, and Connes. In their work, \cite{ji2020mip} estimate the values $m$ and $n$ where proper containment occurs to each be approximately $10^{20}$. It remains of interest to produce smaller values as well as determine the exact values of $m$ and $n$ for which we have proper containment.
 
We contribute to this line of research by applying our results to the theory of quantum correlation sets. Quantum and quantum commuting correlations are induced by the action of states on \emph{projection-valued measures}. Introducing the notion of an \emph{abstract projection} in an operator system, the first two authors gave a new characterization of the set of quantum commuting correlations $C_{qc}(n,m)$ using only the data of certain finite-dimensional operator systems \cite[Theorem 6.3]{araiza2020abstract}. This result relies on an equivalence between abstract projections in operator systems and projections in the C*-envelope of an operator system (Theorem \ref{thm: abstract projections in operator systems}). In this paper, we show that the set of quantum correlations $C_q(n,m)$ can be characterized using only the data of finite-dimensional $k$-Archimedean order unit spaces (Theorem~\ref{thm: quantum correlation characterization}). Crucial to this result is the existence of a faithful representation of the C*-envelope of a $k$-Archimedean order unit space as a C*-subalgebra of a direct sum of matrix algebras of size bounded above by $k$ (Theorem~\ref{thm: k-minimal C*algebra is a subalgebra of ell infinity of matrix algebras}). Our results yield a reformulation of Tsirelson's conjecture in terms of operator systems and $k$-Archimedean order unit spaces (Corollary~\ref{cor: operator system Tsirleson}).

This paper is organized in the following way: In Section~\ref{sec: prelims} we present preliminary results and necessary definitions. In Section~\ref{sec: k-AOU spaces} we examine the category of $k$-Archimedean order unit spaces and its properties. In Section~\ref{sec: injectivity} we present our results on the injective envelope of a $k$-Archimedean order unit space, and in Section~\ref{sec: k minimal Cstar algebras} we use these results to construct the C*-envelope in the category of $k$-Archimedean order unit spaces. In Section~\ref{sec: projections in k AOU spaces} we develop the notion of projections in $k$-Archimedean order unit spaces.  Finally, in Section~\ref{sec: quantum correlations and k-AOU spaces} we use our results to characterize finite-dimensional quantum correlations in terms of states on finite-dimensional $k$-AOU spaces.

\section{Preliminaries}\label{sec: prelims}

A \emph{$*$-vector space} is a complex vector space $\mathcal V$ together with an conjugate-linear involution $*: \mathcal{V} \to \mathcal{V}$. An element $x \in \mathcal V$ such that $x^* = x$ is called \emph{hermitian} and we denote the real subspace of all hermitian elements of $\mathcal{V}$ by $\mathcal V_h$.

If $\mathcal{V}$ is a $*$-vector space, a \emph{cone} is a subset $C \subseteq \mathcal{V}_h$ with $\alpha C \subseteq C$ for all $\alpha \in [0, \infty)$ and such that $C + C \subseteq C.$ We will say the cone $C$ is \emph{proper} if $C \cap -C = \{ 0 \}$.

An \emph{ordered $*$-vector space} $(\mathcal{V},C)$ consists of a $*$-vector space $\mathcal{V}$ with a proper cone $C$.  For any ordered vector space $(\mathcal{V},C)$ we may define a partial order on $\mathcal{V}$ by $v \leq w$ (equivalently $w \geq v$) if and only if $w-v \in C$.  With this partial order, $C$ is exactly the set of non-negative elements of $\mathcal{V}$.

If $(\mathcal{V},C)$ is an ordered $*$-vector space, an element $e \in \mathcal{V}$ is called an \emph{order unit} if for all $v \in \mathcal{V}_h$ there exists $r > 0$ such that $re \geq v$.  An order unit $e$ is called \emph{Archimedean} if whenever $re+v \geq 0$ for all real $r >0$, then $v \geq 0$.  An \emph{Archimedean order unit space} (or \emph{AOU space} for short) is a triple $(\mathcal{V},C,e)$ such that $(\mathcal{V},C)$ is an ordered $*$-vector space and $e$ is an Archimedean order unit for $(\mathcal{V},C)$. 

If $e$ is an order unit for $(\mathcal{V},C)$, the \textit{Archimedean closure} of $C$ is defined to be the set of $x \in \mathcal{V}_h$ with the property that $re + x \in C$ for all $r > 0$. In general the Archimedean closure of a proper cone $C$ may not be proper.

If $\mathcal{V}$ is a complex vector space, then for any $n \in \mathbb{N}$ the vector space of $n \times n$ matrices with entries in $\mathcal{V}$ is denoted $M_n(\mathcal{V})$.  We see that $M_n(\mathcal{V})$ inherits a $*$-operation by $(a_{i,j})_{i,j}^* = (a_{j,i}^*)_{i,j}$.

Let $\mathcal{V}$ be a $*$-vector space.  A family of \emph{matrix cones} $\{ C_n \}_{n=1}^\infty$ is a collection such that $C_n$ is a proper cone of $M_n(\mathcal{V})$ for all $n \in \mathbb{N}$.  We call a family of matrix cones $\{ C_n \}_{n=1}^\infty$ a \emph{matrix ordering} if $$\alpha^* C_n \alpha \subseteq C_m$$ for all $\alpha \in M_{n,m} (\mathbb{C})$. We often use a bold calligraphic symbol such as $\mathbfcal{C}$ to denote a matrix ordering; i.e. $\mathbfcal{C} := \{C_n\}_{n=1}^\infty$. When $\mathbfcal{C}$ is a matrix ordering, we let $\mathbfcal{C}_n$ denote the $n\textsuperscript{th}$ matrix cone of the matrix ordering.

If $x \in \mathcal{V}$, for every $n \in \mathbb{N}$ we define $$x_n := I_n \otimes x = \left( \begin{smallmatrix} x & &  \\ & \ddots & \\  & & x \end{smallmatrix} \right) \in M_n(\mathcal{V}).$$  An \emph{operator system} is a triple $(\mathcal{V}, \mathbfcal{C}, e)$ consisting of a $*$-vector space $\mathcal{V}$, a matrix ordering $\mathbfcal{C}$ on $\mathcal{V}$, and an element $e \in \mathcal{V}$ such that $(\mathcal{V}, \mathbfcal{C}_n, e_n)$ is an AOU space for all $n \in \mathbb{N}$. In this case, we call $e$ an \textit{Archimedean matrix order unit}. If we only have that $e$ is an order unit for each $(\mathcal{V}, \mathcal{C}_n)$, then we call $e$ a \textit{matrix order unit}. We often let $\mathcal{V}$ denote the operator system $(\mathcal{V}, \mathbfcal{C}, e)$ when the unit and matrix ordering are unspecified or clear from context.

If $(\mathcal{V},C)$ and $(\mathcal W,D)$ are ordered $*$-vector spaces, a linear map $\phi : \mathcal{V} \to \mathcal{W}$ is called \emph{positive} if $\phi(C) \subseteq D$.  A positive linear map $\phi : \mathcal{V} \to \mathcal{W}$ is an order isomorphism if $\phi$ is a bijection and $\phi(C) = D$. An injective map $\phi: \mathcal{V} \to \mathcal{W}$ is called an \textit{order embedding} if it is an order isomorphism onto its range.

If $\mathcal{V}$ and $\mathcal{W}$ are $*$-vector spaces and $\phi : \mathcal{V} \to \mathcal{W}$ is a linear map, then for each $n \in \mathbb{N}$ the map $\phi$ induces a linear map $\phi_n : M_n(\mathcal{V}) \to M_n(\mathcal{W})$ by $\phi_n ( (a_{i,j})_{i,j}) = ( \phi(a_{i,j}) )_{i,j}$.  If $(\mathcal{V}, \mathbfcal{C}, e)$ and $(\mathcal{W}, \mathbfcal{D}, f)$ are operator systems, a linear map $\phi : \mathcal{V} \to \mathcal{W}$ is called \emph{completely positive} if $\phi_n(\mathbfcal{C}_n) \subseteq \mathbfcal{D}_n$ for all $n \in \mathbb{N}$.  A completely positive $\phi : \mathcal{V} \to \mathcal{W}$ is called \emph{unital} if $\phi(e) = f$.  A unital completely positive map $\phi : \mathcal{V} \to \mathcal{W}$ is called an  \emph{complete order isomorphism} if $\phi$ is a bijection and $\phi (\mathbfcal{C}_n) = \mathbfcal{D}_n$ for all $n \in \mathbb{N}$. A linear map $\phi: \mathcal{V} \to \mathcal{W}$ is called a \emph{complete order embedding} if $\phi$ is a complete order isomorphism onto its range.

The \emph{C*-envelope} of an operator system $\mathcal{V}$ consists of a C*-algebra $C^*_\textnormal{e} (\mathcal{V})$ and a unital complete order embedding $\kappa : \mathcal{V} \to C^*_\textnormal{e} (\mathcal{V})$ such that the C*-algebra generated by $\kappa(\mathcal{V})$ equals $C^*_\textnormal{e}(\mathcal{V})$ (i.e. $C^*(\kappa(\mathcal{V})) = C^*_\textnormal{e} (\mathcal{V})$) and satisfying the following universal property: whenever $j : \mathcal {V} \to \mathcal{A}$ is a unital complete order embedding with $\mathcal{A} = C^*(j(\mathcal{V}))$, then there exists a unique $*$-epimorphism $\pi : \mathcal{A} \to C^*_\textnormal{e}(\mathcal{V})$ such that $\pi \circ j = \kappa$.  Hamana proved that the C*-envelope of an operator system always exists \cite{hamana1979injective}.

\section{$k$-Archimedean order unit spaces} \label{sec: k-AOU spaces}

In this section, we will define the category of $k$-Archimedean order unit spaces, which will be the natural setting for most of our results. We begin by defining the objects of this category.

\begin{definition}[$k$-Archimedean order unit space]\label{defn: k-aou space}
For any $k \in \mathbb{N}$, a \emph{$k$-Archimedean order unit space} (or $k$-AOU space, for short) is a triple $( \mathcal{V}, C, e)$ consisting of
\begin{itemize}
\item[(i)] $\mathcal{V}$, a $*$-vector space, 
\item[(ii)] $C \subseteq M_k(\mathcal{V})_h$, a proper cone, \emph{compatible} in the sense that for each $\alpha \in M_{k} (\mathbb{C})$, we have $\alpha^* C \alpha \subseteq C$, and
\item[(iii)] $e \in \mathcal{V}$ with the property that $e_k := I_k \otimes e$ is an Archimedean order unit for $(M_k(\mathcal{V}), C)$.  
\end{itemize}
A pair $(\mathcal{V}, C)$ satisfying conditions $(i)$ and $(ii)$ is called a \emph{$k$-ordered $*$-vector space}, and an element $e$ satisfying condition $(iii)$ is called a \emph{$k$-Archimedean order unit} for the $k$-ordered vector space $(\mathcal{V},C)$.
\end{definition}

Next, we define the appropriate morphisms in the category of $k$-AOU spaces.

\begin{definition}[$k$-positive maps]
Let $k \in \mathbb{N}$, and suppose $(\mathcal{V},C)$ and $(\mathcal{W},D)$ are $k$-ordered $*$vector spaces. A linear map $\phi: \mathcal{V} \to \mathcal{W}$ is called \emph{$k$-positive} if $\phi_k(C) \subseteq D$. If $\phi$ is $k$-positive and injective with $\phi_k^{-1}(D) \subseteq C$, then $\phi$ is called a \emph{$k$-order embedding}. A bijective $k$-order embedding is called a \emph{$k$-order isomorphism}.
\end{definition}

In the case when $k=1$, it is clear our notion of a $k$-AOU space is identical to that of an AOU space, and that $k$-positive maps, $k$-order embeddings, and $k$-order isomorphisms are just positive maps, order embeddings, and order isomorphisms, respectively.

In \cite{PaulsenTodorovTomfordeOpSysStructures}, a variety of operator system structures were considered for AOU spaces. In the following, we extend that study to the setting of $k$-AOU spaces.

\begin{definition}[Operator System Structure]
Let $k \in \mathbb{N}$, and suppose $(\mathcal{V}, C, e)$ is a $k$-AOU space. If $\mathbfcal{C}$ is an Archimedean closed matrix ordering on $\mathcal{V}$ satisfying $\mathbfcal{C}_k = C$, then we say $\mathbfcal{C}$ \textit{extends} $C$ or is an \emph{extension} of $C$, and we call the operator system $(\mathcal{V}, \mathbfcal{C}, e)$ an \textit{operator system structure} on $(\mathcal{V}, C, e)$.
\end{definition}

We will now give our first example of an operator system structure on a $k$-AOU space. We first provide the definition, and then prove that our definition yields an operator system structure in the subsequent proposition.

\begin{definition}[The $k$-minimal operator system structure on a $k$-AOU space]
Given a $k$-AOU space $(\mathcal{V}, C, e)$, we define
\[ C_n^\text{$k$-min} := \{ x \in M_n(\mathcal{V})_h : \alpha^* x \alpha \in C \text{ for all } \alpha \in M_{n,k} \} \]
for each $n \in \mathbb{N}$. We let $C^\text{$k$-min} := \{ C_n^\text{$k$-min} \}_{n=1}^\infty$ and  $\mathcal{V}_\text{$k$-min} := (\mathcal{V}, C^\text{$k$-min}, e)$.
\end{definition}

For the proposition that follows, We will use the fact that given a matrix ordering $\mathbfcal C$ on a $*$-vector space $\mathcal V$, then if $\mathbfcal C_1$ is proper, then the matrix ordering is proper. We have included the proof for completeness. 

\begin{lemma}[\cite{araiza2021universal}] \label{lem: Proper matrix cone}
Let $\mathbfcal{C}$ be a matrix ordering on a $*$-vector space $\mathcal{V}$, and suppose that $\mathbfcal{C}_1$ is proper. Then $\mathbfcal{C}_n$ is proper for every $n \in \N$.
\end{lemma}

\begin{proof}
Suppose that $x = (x_{ij}) \in \mathbfcal{C}_n \cap -\mathbfcal{C}_n$. Then for each $k= 1, 2, \dots, n$ we have $x_{kk} = e_k^* x e_k \in \mathbfcal{C}_1 \cap -\mathbfcal{C}_1$, where $e_k \in \mathbb{C}^n$ is the $k$\textsuperscript{th} standard unit vector. Hence $x_{kk} = 0$ since $\mathbfcal C_1$ is proper. Let $k,l \in \{1,2, \dots, n\}$ with $k \neq l$. Then $(e_k + e_l)^* x (e_k + e_l) = x_{lk} + x_{kl} = 2\operatorname{Re}(x_{lk}) \in \mathbfcal{C}_1 \cap - \mathbfcal{C}_1$. Hence $\operatorname{Re}(x_{lk}) = 0$. Also $(e_k - ie_l)^* x (e_k - ie_l) = i(x_{lk} - x_{kl}) = 2i\operatorname{Im}(x_{lk}) \in \mathbfcal{C}_1 \cap - \mathbfcal{C}_1$. Hence $\operatorname{Im}(x_{lk}) = 0$. Thus $x_{lk} = 0$, and it follows that $x = 0$ and $\mathbfcal{C}_n$ is proper.
\end{proof}

\begin{proposition}
Let $k \in \mathbb{N}$ and suppose that $(\mathcal{V}, C, e)$ is a $k$-AOU space. Then $C = C_k^\text{$k$-min}$ and $(\mathcal{V}, C^\text{$k$-min}, e)$ is an operator system.
\end{proposition}

\begin{proof}
It is clear that $C = C_k^\text{$k$-min}$ since $\alpha^* C \alpha \subseteq C$ for every $\alpha \in M_k$. It remains to show $(\mathcal{V}, C^\text{$k$-min}, e)$ is an operator system. We first verify $C^\text{$k$-min}$ is a matrix ordering. Suppose $x \in C_n^\text{$k$-min}$ and $\beta \in M_{n,m}$. Then for every $\alpha \in M_{m,k}$ we have that \[ \alpha^* \beta^* x \beta \alpha = (\beta \alpha)^* x (\beta \alpha) \in C. \] Thus $\beta^* x \beta \in C_m^\text{$k$-min}$. Next suppose $t > 0$ and $x,y \in C_n^\text{$k$-min}$. Then for all $\alpha \in M_{n,k}$ we have $\alpha^*(x + ty) \alpha = \alpha^* x \alpha + t \alpha^* y \alpha \in C$. Thus, $C_n^\text{$k$-min}$ is a cone for each $n$. Therefore $C^\text{$k$-min}$ is a matrix ordering. The fact that $C^\text{$k$-min}$ is proper is immediate from Lemma~\ref{lem: Proper matrix cone} and since the initial cone $C$ is proper. 

Next, we show $e$ is an Archimedean matrix order unit. To see that $e$ is a matrix order unit, it suffices to check that $e$ is an order unit. For a proof of this fact see e.g. \cite[Proposition 2.4]{araiza2020abstract}. Let $x \in \mathcal{V}_h$. Then there exists $t > 0$ such that $I_k \otimes x + t I_k \otimes e \in C$. Hence, for each $\alpha \in M_{1,k}$ \[ \alpha^*(x + te)\alpha = (\alpha^* \alpha) \otimes x + t (\alpha^* \alpha) \otimes e = \alpha^* (I_k \otimes x + t I_k \otimes e) \alpha \in C. \]
It follows that $x + te \in C_{1}^\text{$k$-min}$. So $e$ is an order unit and thus a matrix order unit. Finally, we verify the Archimedean property. Suppose $x + tI_n \otimes e \in C_n^\text{$k$-min}$ for every $t > 0$. Then for every $\alpha \in M_{n,k}$ with $\alpha \neq 0$ and every $t > 0$,
\[ \alpha^* x \alpha + \frac{t}{\|\alpha\|^2} \alpha^*(I_n \otimes e) \alpha \in C. \]
But \[ \alpha^* x \alpha + \frac{t}{\|\alpha\|^2} \alpha^*(I_n \otimes e) \alpha \leq \alpha^* x \alpha + t I_k \otimes e \] and hence $\alpha^* x \alpha + t I_k \otimes e \in C$ for every $t > 0$. Since $I_k \otimes e$ is Archimedean for $C$, $\alpha^* x \alpha \in C$. We conclude $x \in C_n^\text{$k$-min}$ and therefore $e$ is Archimedean.
\end{proof}

The next proposition gives the basic properties of the operator system structure $\mathcal{V}_\text{$k$-min}$.

\begin{proposition} \label{prop: properties of V_k-min}
Let $k \in \mathbb{N}$, and suppose that $(\mathcal{V}, C, e)$ is a $k$-AOU space. If $\mathbfcal{C}$ is an extension of $C$, then $\mathbfcal{C}_n \subseteq C_n^\text{$k$-min}$ for every $n \in \mathbb{N}$. Moreover, $\mathbfcal{C}_n = C_n^\text{$k$-min}$ for every $n \in \{1,2,\dots,k\}$. Hence, the identity map is a unital completely positive $k$-order embedding from $(\mathcal{V}, \mathbfcal{C}, e)$ to $\mathcal{V}_\text{$k$-min}$.
\end{proposition}

\begin{proof}
Let $n \in \mathbb N$ and let $x \in \mathbfcal{C}_n$. Since $\mathbfcal{C}$ is a matrix ordering on $\mathcal{V}$, $\alpha^* x \alpha \in \mathbfcal{C}_k = C$ for every $\alpha \in M_{n,k}$. Thus $\mathbfcal{C}_n \subseteq C_n^\text{$k$-min}$.

Now suppose $n < k$ and let $x \in C_n^\text{$k$-min}$. Then $x \oplus 0_{k-n} \in C = \mathbfcal{C}_k$. Since $\mathbfcal{C}$ is a matrix ordering, $x \in \mathbfcal{C}_n$, since \[ x = \begin{bmatrix} I_{n} \\ 0_{n,k-n} \end{bmatrix}^* (x \oplus 0_{k-n}) \begin{bmatrix} I_{n} \\ 0_{n,k-n} \end{bmatrix}. \]
The statement follows. \end{proof}

Later, we will make use of the relations between $k$-positive maps on $k$-AOU spaces and completely positive maps on operator systems of the form $\mathcal{V}_\text{$k$-min}$. The next two results summarize these relationships.

\begin{proposition} \label{prop: k-positive implies cp on k-min}
Suppose $\varphi: \mathcal{V} \to \mathcal{W}$ is a $k$-positive map between $k$-AOU spaces $(\mathcal{V},C,e)$ and $(\mathcal{W},D,e)$. Then $\varphi: \mathcal{V} \to \mathcal{W}_\text{$k$-min}$ is completely positive with respect to any operator system structure on $\mathcal{V}$.
\end{proposition}

\begin{proof}
Let $\mathbfcal{C}$ be an extension of $C$ and suppose $x \in \mathbfcal{C}_n$. Then $\alpha^* x \alpha \in C$ for every $\alpha \in M_{n,k}$. Since $\varphi$ is $k$-positive, $\varphi_k(\alpha^* x \alpha) \in D$ for every $\alpha \in M_{n,k}$. However, $\varphi_k(\alpha^* x \alpha) = \alpha^* \varphi_n(x) \alpha$. It follows $\varphi_n(x) \in D_n^\text{$k$-min}$, and thus $\varphi$ is completely positive. \qedhere
\end{proof}

\begin{corollary} \label{cor: k-min is injective}
Let $\mathcal V$ and $\mathcal W$ be k-AOU spaces and suppose $i: \mathcal{V} \to \mathcal{W}$ is a $k$-order embedding. Then $i: \mathcal{V}_\text{$k$-min} \to \mathcal{W}_\text{$k$-min}$ is a complete order embedding.
\end{corollary}

\begin{proof}
By Proposition~\ref{prop: k-positive implies cp on k-min}, $i: \mathcal{V}_\text{$k$-min} \to \mathcal{W}_\text{$k$-min}$ is completely positive. Since $i^{-1}: i(\mathcal{V}) \to \mathcal{V}$ is $k$-positive, and since $i(\mathcal{V}_\text{$k$-min})$ induces an operator system structure on $\mathcal{V}$, $i^{-1}: i(\mathcal{V}_\text{$k$-min}) \to \mathcal{V}_\text{$k$-min}$ is completely positive by Proposition~\ref{prop: k-positive implies cp on k-min}. Thus $i$ is a complete order embedding.
\end{proof}

Our notation and terminology for $\mathcal{V}_\text{$k$-min}$ is closely related to concepts introduced by Xhabli in \cite{xhabli2012super}. We will recall some definitions and results due to Xhabli and explain the connection between Xhabli's concepts and ours in Remark~\ref{Xhabli-relationship-rem}.

\begin{definition}[$\text{OMIN}_k(\mathcal{V})$]
Let $\mathcal{V}$ be an operator system. For each $n \in \mathbb{N}$, let
\[ C_n^\text{$k$-min}(\mathcal{V}) := \{ x \in M_n(\mathcal{V}) : \phi_n(x) \geq 0 \text{ for every ucp } \phi: \mathcal{V} \to M_k(\mathbb C). \} \]
We let $\text{OMIN}_k(\mathcal{V})$ denote the triple $(\mathcal{V}, \{ C_n^\text{$k$-min}(\mathcal{V}) \}_{n=1}^\infty, e)$.
\end{definition}

Xhabli proved the following theorem detailing the basic properties of $\text{OMIN}_k(\mathcal{V})$.

\begin{theorem}(\cite[Theorem 3.7]{xhabli2012super}) \label{thm: Xhabli's theorem}
Let $\mathcal{V}$ be an operator system, and let $k \in \mathbb{N}$. Then $\text{OMIN}_k(\mathcal{V})$ is an operator system. Moreover, the identity map $id: \mathcal{V} \to \text{OMIN}_k(\mathcal{V})$ is a unital $k$-order embedding. In particular, $\id: \mathcal V \to \text{OMIN}_k(\mathcal V)$ is completely positive.
\end{theorem}

We remark that Xhabli's $\text{OMIN}_k(\mathcal{V})$ is constructed from an operator system $\mathcal{V}$, whereas our $\mathcal{V}_\text{$k$-min}$ requires only a $k$-AOU space as its input. Nonetheless, we show below that the two notions are equivalent.

\begin{definition}[The $k$-state space]
Given a $k$-AOU space $(\mathcal{V}, C, e)$ we define 
$$\mathcal{S}_k (\mathcal{V}) := \{ \phi : \mathcal{V} \to M_k(\mathbb{C}) : \text{ $\phi$ is unital $k$-positive} \}$$
and we call the elements of $\mathcal{S}_k(\mathcal{V})$ the \emph{$k$-states} on $( \mathcal{V}, C, e)$.
\end{definition}

\begin{proposition} \label{prop: k-min and OMIN_k}
Let $k \in \mathbb{N}$ and suppose $(\mathcal{V}, C, e)$ is a $k$-AOU space. Suppose that $x \in M_n(\mathcal{V})$ for some $n \in \mathbb{N}$. Then the following statements are equivalent:
\begin{enumerate}
    \item $x \in C_n^\text{$k$-min}$.
    \item $\phi_n(x) \in M_{kn}^+$ for every $\phi \in \mathcal{S}_k(\mathcal{V})$.
    \item $\psi_n(x) \in M_{kn}^+$ for every $k$-positive linear map $\psi: \mathcal{V} \to M_k$.
\end{enumerate}
\end{proposition}

\begin{proof}
We will demonstrate the equivalence of (1) and (2). The equivalence of (2) and (3) follows from \cite[Proposition 3.4]{xhabli2012super}.

By Proposition \ref{prop: properties of V_k-min}, the identity map $id: \text{OMIN}_k(\mathcal{V}_\text{$k$-min}) \to \mathcal{V}_\text{$k$-min}$ is unital completely positive, since $\{ C_n^\text{$k$-min}(\mathcal{V}_\text{$k$-min}) \}_{n=1}^\infty$ is an extension of $C$. However, by Theorem \ref{thm: Xhabli's theorem}, $id^{-1}: \mathcal{V}_\text{$k$-min} \to \text{OMIN}_k(\mathcal{V}_\text{$k$-min})$ is also unital completely positive. It follows that $C_n^\text{$k$-min}(\mathcal{V}_\text{$k$-min}) = C_n^\text{$k$-min}$. The equivalence follows from these observations. \qedhere

\end{proof}

\begin{remark} \label{Xhabli-relationship-rem}
Proposition \ref{prop: k-min and OMIN_k} shows that our notion of $\mathcal{V}_\text{$k$-min}$ agrees with Xhabli's notion of $\text{OMIN}_k(\mathcal{V})$ in the following sense. Let $\text{OpSys}$ denote the category whose objects are operator systems and whose morphisms are unital completely positive maps, and, for $k \in \mathbb{N}$, let $\text{$k$-AOU}$ denote the category whose objects are $k$-AOU spaces and whose morphisms are unital $k$-positive maps. Let $F: \text{OpSys} \to \text{$k$-AOU}$ denote the functor that maps $(\mathcal{V}, \mathbfcal{C}, e)$ to $(\mathcal{V}, \mathbfcal{C}_k, e)$ and leaves morphisms unchanged (i.e., the forgetful functor). Likewise, we may regard $(\cdot)_\text{$k$-min}$ as a functor from $\text{$k$-AOU}$ to $\text{OpSys}$, and we may regard $\text{OMIN}_k$ as a functor from $\text{OpSys}$ to itself, each acting in the obvious way on objects and as the identity on morphisms.  Proposition~\ref{prop: k-min and OMIN_k} implies that the diagram 
\[ \begin{tikzcd}
\text{OpSys} \arrow[rd, "\text{OMIN}_k" '] \arrow[r, "F"] & \text{$k$-AOU} \arrow[d, "(\cdot)_\text{$k$-min}"] \\
& \text{OpSys}
\end{tikzcd}
\]
commutes. Moreover, Proposition \ref{prop: k-positive implies cp on k-min} and Corollary \ref{cor: k-min is injective} imply $(\cdot)_\text{$k$-min}$ maps unital $k$-positive maps and unital $k$-order embeddings to unital completely positive maps and unital complete order embeddings.
\end{remark}

We end this section by writing down the analogue of Xhabli's $\text{OMAX}_k$ functor for the category of $k$-AOU spaces and recording its properties.

\begin{definition}[$\text{OMAX}_k(\mathcal{V})$]
Let $(\mathcal{V}, \mathbfcal{C}, e)$ be an operator system. For each $n \in \mathbb{N}$, let
\[ D_n^\text{$k$-max}(\mathcal{V}) := \{ \alpha^* \text{diag}(s_1, \dots, s_m) \alpha: \alpha \in M_{k,n} \text{ and } s_1, \dots, s_m \in \mathbfcal{C}_k \}, \]
and let $C_n^\text{$k$-max}(\mathcal{V})$ denote the Archimedean closure of $D_n^\text{$k$-max}$. We use the notation $\text{OMAX}_k(\mathcal{V})$ to denote the triple $(\mathcal{V}, C_n^\text{$k$-max}(\mathcal{V}), e)$.
\end{definition}

The following theorem of Xhabli details the basic properties of $\text{OMAX}_k(\mathcal{V})$.

\begin{theorem}[Xhabli] \label{thm: Xhabli's OMAX properties}
Let $\mathcal{V}$ be an operator system, and let $k \in \mathbb{N}$. Then $\text{OMAX}_k(\mathcal{V})$ is an operator system. Moreover, the identity map $\text{id}: \text{OMAX}_k(\mathcal{V}) \to \mathcal{V}$ is a unital $k$-order embedding. In particular, the map $\text{id}: \text{OMAX}_k(\mathcal{V}) \to \mathcal{V}$ is completely positive.
\end{theorem}

\begin{remark} \label{rmk: OMAX = V_kmax}
Given a $k$-AOU space $(\mathcal{V}, C, e)$, we may define, for each $n \in \mathbb{N}$, cones \[ D_n^\text{$k$-max} := \{ \alpha^* \text{diag}(s_1, \dots, s_m) \alpha: \alpha \in M_{k,n} \text{ and } s_1, \dots, s_m \in C \}, \]
and let $C_n^\text{$k$-max}$ denote the Archimedean closure of $D_n^\text{$k$-max}$. If $\mathcal{W}$ is an operator system structure on the $k$-AOU space $\mathcal{V}$, then it is clear that $C_n^\text{$k$-max}(\mathcal{W}) = C_n^\text{$k$-max}$. It follows from Theorem \ref{thm: Xhabli's OMAX properties} that \[ \mathcal{V}_\text{$k$-max} := (\mathcal{V}, C^\text{$k$-max}, e) \] is an operator system with the property that the identity map $\text{id}: \mathcal{V}_\text{$k$-max} \to \mathcal{V}$ is a unital $k$-order embedding for any operator system structure on the range. In particular, $\text{id}: \mathcal{V}_\text{$k$-max} \to \mathcal{V}$ is completely positive. 
\end{remark}

\begin{corollary} \label{cor: extensions of k-AOU spaces agree on levels preceeding k}
Let $(\mathcal{V}, C, e)$ be a $k$-AOU space, and suppose that $\mathbfcal{C}$ is an extension of $C$. Then $\mathbfcal{C}_n = C_n^\text{$k$-min} = C_n^\text{$k$-max}$ for every $n \in \{1,2,\dots,k\}$.
\end{corollary}

As a consequence of these observations, we see for a given $k$-AOU space $(\mathcal{V}, C, e)$, there is a unique positive cone $C_n \subseteq M_n(\mathcal{V})_h$ for $n \in \{1,2,\dots,k\}$ that agrees with the positive cone for any operator system structure on $\mathcal{V}$. Consequently, given a $k$-AOU space $\mathcal{V}$, we can talk about the positive elements of $M_n(\mathcal{V})$ for any $n \in \{1,2,\dots,k\}$ without ambiguity. In particular, we can describe the positive elements of $\mathcal{V}$. 

\begin{corollary} \label{cor: positive at 1st level}
Let $(\mathcal{V},C,e)$ be a $k$-AOU space, and suppose that $\mathbfcal{C}$ is an extension of $C$. Then $x \in \mathbfcal{C}_1$ if and only if $I_k \otimes x \in C$.
\end{corollary}

\section{The injective envelope for a $k$-AOU space}\label{sec: injectivity}

The goal of this section is to prove the existence of an injective envelope in the category of $k$-AOU spaces and to discuss its properties. We first recall the definition of the injective envelope in the category $\text{OpSys}$, and then define what is meant by an injective envelope in the category $\text{$k$-AOU}$.

\begin{definition}
An operator system $\mathcal{T}$ is called \textit{injective} if whenever $\mathcal{V} \subseteq \widetilde{\mathcal{V}}$ is an inclusion of operator systems and $\varphi: \mathcal{V} \to \mathcal{T}$ is a unital completely positive map, there exists a unital completely positive extension $\widetilde{\varphi}: \widetilde{\mathcal{V}} \to \mathcal{T}$; i.e. $\widetilde{\varphi}$ is unital completely positive, and its restriction to $\mathcal{V}$ is equal to $\varphi$.
\end{definition}

In his seminal work \cite{Arveson69}, Arveson proved that whenever $H$ is a Hilbert space, the C*-algebra $B(H)$ is an injective operator system. By the representation theorem of Choi and Effros in \cite{choi1977injectivity}, every operator system $\mathcal{V}$ arises as a subsystem of $B(H)$ for some Hilbert space $H$. Therefore every operator system may be regarded as a subsystem of an injective operator system. A remarkable result of Hamana \cite{hamana1979injective} implies that every operator system is contained in a canonical ``smallest'' injective operator system called the injective envelope.

\begin{definition}
Let $\mathcal{V}$ be an operator system. An \textit{injective envelope} for $\mathcal{V}$ is a pair $(\mathcal{I}, \kappa)$ satisfying the following properties:
\begin{enumerate}
    \item $\mathcal{I}$ is injective.
    \item $\kappa: \mathcal{V} \to \mathcal{I}$ is a unital complete order embedding.
    \item If $\mathcal{W}$ is another operator system and $\varphi: \mathcal{I} \to \mathcal{W}$ restricts to a unital complete order embedding on $\kappa(\mathcal{V})$, then $\varphi$ is a unital complete order embedding on $\mathcal{I}$.
\end{enumerate}
\end{definition}

\begin{remark}
If $\kappa: \mathcal{V} \to \mathcal{I}$ satisfies conditions (2) and (3) in the preceding definition, then $(\mathcal{I}, \kappa)$ is called an \textit{essential extension}. Thus, an injective envelope is an injective essential extension. Hamana proved that the injective envelope is a minimal injective extension, in the sense that if $\mathcal{W}$ is injective and $\kappa(\mathcal{V}) \subseteq \mathcal{W} \subseteq \mathcal{I}$, then $\mathcal{W} = \mathcal{I}$.
\end{remark}

It was shown by Hamana in \cite{hamana1979injective} that every operator system has an injective envelope, denoted $I(\mathcal{V})$, and that this injective envelope is unique up to unital complete order isomorphism.

We now extend the notions of ``injective'' and ``injective envelope'' to the category $\text{$k$-AOU}$.

\begin{definition}
A $k$-AOU space $\mathcal{T}$ is called \textit{injective} if whenever $\mathcal{V} \subseteq \widetilde{\mathcal{V}}$ is an inclusion of $k$-AOU spaces and $\varphi: \mathcal{V} \to \mathcal{T}$ is a unital $k$-positive map, there exists a unital $k$-positive extension $\widetilde{\varphi}: \widetilde{\mathcal{V}} \to \mathcal{T}$; i.e., $\widetilde{\varphi}$ is unital $k$-positive, and its restriction to $\mathcal{V}$ is equal to $\varphi$.
\end{definition}

\begin{definition}

Given a $k$-AOU space $\mathcal{V}$, an \textit{injective envelope} for $\mathcal{V}$ is a pair $(\mathcal{I}, \kappa)$ satisfying the following properties:
\begin{enumerate}
    \item $\mathcal{I}$ is injective.
    \item $\kappa: \mathcal{V} \to \mathcal{I}$ is a unital $k$-order embedding.
    \item If $\mathcal{W}$ is another $k$-AOU space and $\varphi: \mathcal{I} \to \mathcal{W}$ restricts to a unital $k$-order embedding on $\kappa(\mathcal{V})$, then $\varphi$ is a unital $k$-order embedding on $\mathcal{I}$.
\end{enumerate}
\end{definition}

\begin{remark}
When $H$ is a Hilbert space and $\dim(H) > k$, the $k$-AOU space $B(H)$ is not an injective $k$-AOU space (see Remark \ref{rmk: M_n not injective for k-AOU}). However, we will see that $B(H)$ is injective when $\dim(H) \leq k$ (see Lemma \ref{lem: M_k is injective}), and we will demonstrate other examples of injective $k$-AOU spaces as well (see Proposition \ref{rmk: ell infinity is injective relative to operator systems and cp maps}). Since injective operator systems may fail to be injective as $k$-AOU spaces, the existence of an injective envelope for a $k$-AOU space is not immediate from Hamana's Theorem.
\end{remark}

\begin{definition}
Let $k \in \mathbb N$. An operator system $\mathcal V$ is called \emph{$k$-minimal} if $\mathcal{V}$ is completely order isomorphic to $\text{OMIN}_k(\mathcal{V})$.
\end{definition}

For the remainder of this section, given a set $I$, we will let $\mathcal D^I:= \bigoplus_{i \in I} M_k^i$ denote the $\ell_\infty$-direct sum of matrix algebras of size $k$, where $M_k^i = M_k$ for each $i \in I$. The following proposition is a consequence of \cite[Theorem~3.1]{xhabli2012super}.

\begin{proposition}\label{prop: k-minimal system if and only if subsystem of ell infinity}
Let $k \in \mathbb N.$ An operator system $\mathcal V$ is $k$-minimal if and only if there exists a complete order embedding $i: \mathcal V \to \mathcal{D}^I$ for some set $I$.
\end{proposition}

\begin{proof}
By \cite[Theorem 3.1]{xhabli2012super}, an operator system $\mathcal{V}$ is $k$-minimal if and only if there exists a unital complete order embedding $u: \mathcal{V} \to M_k(C(\mathcal{S}_k(\mathcal{V})))$, where $M_k(C(\mathcal{S}_k(\mathcal{V})))$ denotes the set of functions $f: \mathcal{S}_k(\mathcal{V}) \to M_k$ continuous with respect to the weak-$*$ topology on $\mathcal{S}_k(\mathcal{V})$. Let $I = \mathcal{S}_k(\mathcal{V})$, and define $j: M_k(C(\mathcal{S}_k(\mathcal{V}))) \to \mathcal{D}^I$ by defining the $\varphi$\textsuperscript{th} component using
\[ j( f )_{\varphi} = (f(\varphi)) \]
for each $\varphi \in I$. Then $j$ is a unital complete order embedding, and hence $i = j \circ u$ defines a unital complete order embedding $i: \mathcal{V} \to \mathcal{D}^I$.
\end{proof}

\begin{remark} \label{rmk: M_n not injective for k-AOU}
Let $n > k$, and let $\mathcal{V}$ be an operator system. Suppose that $\phi: \mathcal{V} \to M_n$ is $k$-positive. It follows that $\phi: \text{OMIN}_k(\mathcal{V}) \to M_n$ is also $k$-positive. If $M_n$ were injective in the category $k$-AOU, then there would exist a $k$-positive extension $\widetilde{\phi}: \mathcal{D}^I \to M_n$. The map $\widetilde{\phi}$ is completely positive if and only if the restriction to each of its summands $M_k^i$ is completely positive. However, by Theorem 3.14 of \cite{paulsen2002completely}, any $k$-postive map on $M_k$ is completely positive. Thus $\widetilde{\phi}$ would be a completely positive extension of $\phi$. It follows that $\phi$ is completely positive on $\text{OMIN}_k(\mathcal{V})$, and consequently $\phi$ is completely positive on on $\mathcal{V}$. However, this is absurd, since there exist linear maps on $M_n$ which are $k$-positive but fail to be $(k+1)$-positive (see Theorem 1 of \cite{choiPositiveMaps1972}). It follows that $M_n$ is not injective in the category of $k$-AOU spaces. In spite of this, we see below that the situation is different when $n \leq k$.
\end{remark}

\begin{lemma} \label{lem: M_k is injective}
Let $k \in \mathbb N$. Then $M_k$ is injective in the category $k$-AOU. 
\end{lemma}

\begin{proof}
Let $\cal V \subseteq \wt{\cal V}$ be an inclusion of $k$-AOU spaces. By Corollary~\ref{cor: k-min is injective} we have $\cal V_{\tx{$k$-min}} \subseteq \wt{\cal V}_{\tx{$k$-min}}$ as operator systems. Let $\vp: \cal V \to M_k$ be a unital $k$-positive map. Then $\vp$ is $k$-positive as a map from $\cal V_{\tx{$k$-min}}$ to $M_k$ and it necessarily follows that $\vp$ is unital completely positive.  By injectivity of the operator system $M_k$ there exists an extension $\wt{\vp}: \wt{\cal V}_{\tx{$k$-min}} \to M_k$ that is (unital) completely positive. It follows that $\wt{\vp}: \wt{\cal V} \to M_k$ is $k$-positive and necessarily extends the $k$-positive map $\vp.$ 
\end{proof}

\begin{lemma}\label{lem: k-positive maps into ell infinity are cp on k-min}
Let $I$ be a set and let $\cal V$ be a $k$-AOU space. If $u: \cal V \to \mathcal D^I$ is a $k$-positive map, then $u: \cal V \to \mathcal D^I$ is completely positive with respect to any operator system structure on $\mathcal V.$
\end{lemma}

\begin{proof}
This follows from Proposition \ref{prop: k-positive implies cp on k-min} and the fact that $\mathcal{D}^I$ is $k$-minimal.
\end{proof}

\begin{lemma} \label{rmk: ell infinity is injective relative to operator systems and cp maps}
For any set $I$, the operator system $\mathcal{D}^I$ is injective.
\end{lemma}

\begin{proof}
Suppose $\cal V \subseteq \wt{\cal V}$ is an inclusion of operator systems, and let $u: \cal V \to \mathcal D^I$ be a unital completely positive map. Let $u_i: \cal V \to M_k$ be the compression to the $i$\textsuperscript{th} block on the diagonal. The fact $u$ is unital completely positive implies $u_i$ is unital completely positive for each $i \in I$, and by injectivity of $M_k$ there exists a unital completely positive extension $\wt{u}_i: \wt{\cal V} \to M_k$ of $u_i.$ Then $\wt{u}:= \bigoplus_{i \in I} \wt{u}_i: \wt{\cal V} \to \mathcal D^I$ is the desired extension of $u$. 
\end{proof}

\begin{proposition} \label{prop: ell infinity is injective relative to k-AOU spaces and unital k-positive maps}
For any set $I$, the $k$-AOU space $\mathcal D^I$ is injective.
\end{proposition}

\begin{proof}
Let $\cal V \subseteq \wt{\cal V}$ be an inclusion of $k$-AOU spaces and let $u: \cal V \to \mathcal D^I$ be a unital $k$-positive map. By Lemma~\ref{lem: k-positive maps into ell infinity are cp on k-min}  $u: \mathcal V_\text{$k$-min} \to \mathcal D^I$ is unital completely positive. By Corollary~\ref{cor: k-min is injective} and Lemma~\ref{rmk: ell infinity is injective relative to operator systems and cp maps}, there exists a unital completely positive extension $\wt{u}: \wt{\mathcal V}_\text{$k$-min} \to \mathcal D^I$ of $u$. Thus $\wt{u}: \wt{\mathcal V} \to \mathcal D^I$ is the desired extension of $u$.  \qedhere
\end{proof}

\begin{theorem}\label{thm: the injective envelope of k-min is also k-min}
Suppose that $\mathcal{V}$ is a $k$-minimal operator system. Then $I(\mathcal{V})$ is $k$-minimal.
\end{theorem}

\begin{proof}
By Proposition \ref{prop: k-minimal system if and only if subsystem of ell infinity}, there exists a unital complete order embedding $u: \mathcal{V} \to \mathcal{D}^I$ for some set $I$. By Lemma \ref{rmk: ell infinity is injective relative to operator systems and cp maps}, $\mathcal{D}^I$ is an injective operator system.

Letting $\kappa: \mathcal V \to I(\mathcal V)$ denote the natural complete order embedding into the injective envelope, we may identify $\mathcal V$ with $\kappa(\mathcal V)$ and therefore we will assume $\mathcal{V} \subseteq I(\mathcal{V})$. Since $\mathcal{D}^I$ is injective, there exists an extension $\widetilde{u}: I(\mathcal{V}) \to \mathcal{D}^I$ of $u$. Since $u$ is a unital complete order embedding, and since $I(\mathcal{V})$ is an essential extension of $\mathcal{V}$, $\widetilde{u}$ is a unital complete order embedding. It follows from Proposition \ref{prop: k-minimal system if and only if subsystem of ell infinity} that $I(\mathcal{V})$ is $k$-minimal.
\end{proof}

\begin{theorem}\label{thm: existence of injective envelope in k-AOU}
Suppose that $\mathcal{V}$ is a $k$-AOU space. Then there exists a unique injective envelope $I(\mathcal{V})$ in the category $k$-AOU. Moreover, $I(\mathcal{V})$ is $k$-order isomorphic to $I(\mathcal{V}_\text{$k$-min})$.
\end{theorem}

\begin{proof}
Since $\mathcal{V}$ is $k$-order isomorphic to $\mathcal{V}_\text{$k$-min}$ and $\kappa: \mathcal{V}_\text{$k$-min} \to I(\mathcal{V}_\text{$k$-min})$ is a unital complete order embedding, it is clear that $\kappa: \mathcal{V} \to I(\mathcal{V}_\text{$k$-min})$ is $k$-order embedding of $k$-AOU spaces.

We first check that $I(\mathcal{V}_\text{$k$-min})$ is an injective $k$-AOU space. Suppose $\mathcal{W} \subseteq \widetilde{\mathcal{W}}$ is an inclusion of $k$-AOU spaces, and let $u: \mathcal{W} \to I(\mathcal{V}_\text{$k$-min})$ be a unital $k$-positive map. By Corollary \ref{cor: k-min is injective}, $\mathcal{W}_\text{$k$-min} \subseteq \widetilde{\mathcal{W}}_\text{$k$-min}$ is an inclusion of operator systems. Since $u: \mathcal{W}_\text{$k$-min} \to I(\mathcal{V}_\text{$k$-min})$ is unital k-positive, Theorem~\ref{thm: the injective envelope of k-min is also k-min} implies $I(\mathcal V_\text{$k$-min})$ is k-minimal and therefore $u: \mathcal W_\text{$k$-min} \to I(\mathcal V_\text{$k$-min})$ is completely positive by Proposition \ref{prop: k-positive implies cp on k-min}. Since $I(\mathcal{V}_\text{$k$-min})$ is an injective operator system, there exists a unital completely positive extension $\widetilde{u}: \widetilde{\mathcal{W}}_\text{$k$-min} \to I(\mathcal{V}_\text{$k$-min})$. Thus $\widetilde{u}: \widetilde{\mathcal{W}} \to I(\mathcal{V}_\text{$k$-min})$ is a unital $k$-positive extension of $u$, and $I(\mathcal{V}_\text{$k$-min})$ is injective.

We now check that $I(\mathcal{V}_\text{$k$-min})$ is an essential cover for $\mathcal{V}$ as a $k$-AOU space. Suppose $\mathcal{W}$ is a $k$-AOU space and $\varphi: I(\mathcal{V}_\text{$k$-min}) \to \mathcal{W}$ is a unital $k$-positive map whose restriction to $\mathcal{V}$ is a unital $k$-order embedding. Then $\varphi: I(\mathcal{V}_\text{$k$-min}) \to \mathcal{W}_\text{$k$-min}$ is a unital completely positive map whose restriction to $\mathcal{V}_\text{$k$-min}$ is a unital complete order embedding by Corollary~\ref{cor: k-min is injective}. Since $I(\mathcal{V}_\text{$k$-min})$ is an essential extension of $\mathcal{V}_\text{$k$-min}$, $\varphi: I(\mathcal{V}_\text{$k$-min}) \to \mathcal{W}_\text{$k$-min}$ is also a unital complete order embedding. It follows $\varphi: I(\mathcal{V}_\text{$k$-min}) \to \mathcal{W}$ is a unital $k$-order embedding. Therefore $I(\mathcal{V}_\text{$k$-min})$ is an essential injective cover for $\mathcal{V}$, and hence an injective envelope for $\mathcal{V}$.

Finally, we verify uniqueness. Suppose $(\mathcal{I}, \kappa')$ is another injective envelope for $\mathcal{V}$ in the category $k$-AOU. Without loss of generality, we may assume that $\kappa'$ is the inclusion map. By the injectivity of $I(\mathcal{V}_\text{$k$-min})$, the unital $k$-order embedding $\kappa: \mathcal{V} \to I(\mathcal{V}_\text{$k$-min})$ extends to a unital $k$-positive map $\widetilde{\kappa}: \mathcal{I} \to I(\mathcal{V}_\text{$k$-min})$. Since $\mathcal{I}$ is an essential extension of $\mathcal{V}$, $\widetilde{\kappa}$ is a unital $k$-order embedding as well. By the injectivity of $\mathcal{I}$, the unital $k$-order embedding $\widetilde{\kappa}^{-1}: \widetilde{\kappa}(\mathcal{I}) \to \mathcal{I}$ extends to a unital $k$-positive map $\rho: I(\mathcal{V}_\text{$k$-min}) \to \mathcal{I}$. Since $\rho$ extends the unital $k$-order embedding $\kappa^{-1}: \kappa(\mathcal{I}) \to \mathcal{I}$, $\rho$ is a unital $k$-order embedding by once again invoking the fact that $I(\mathcal V_\text{$k$-min})$ is essential. It follows that $\rho = \widetilde{\kappa}^{-1}$ and hence $\widetilde{\kappa}$ is a unital $k$-order isomorphism.  Therefore $I(\mathcal{V}_\text{$k$-min})$ is unique up to unital $k$-order isomorphism. 
\end{proof}

\begin{remark}
Let $\mathcal V$ be a $k$-AOU space. Following the standard terminology, a linear map $\vp: \mathcal D^I \to \mathcal D^I$ will be called a \emph{$\mathcal V$-map} if $\vp$ is $k$-positive and $\vp \vert_\mathcal V = \id_\mathcal V.$ A $\mathcal V$-map $\vp$ is called a \emph{$\mathcal V$-projection} if $\vp \circ \vp = \vp.$  If $\vp$ is a $\mathcal V$-map then the map $p: \mathcal D^I \to [0,\infty)$ defined as $p(x):= \norm{\vp(x)}$ is called a \emph{$\mathcal V$-seminorm}. We partially order $\mathcal V$-seminorms by declaring $p_\vp \preceq p_\psi$ if and only if $p_\vp(x) \leq p_\psi(x)$ for all $x \in \mathcal D^I.$ We partially order $\mathcal V$-projections by declaring $\vp \preceq \psi$ if and only $\vp \psi = \psi \vp = \vp.$ Thus, given a $k$-AOU space $\mathcal V$, and using results of this section, one may proceed in constructing the injective envelope of $\mathcal V$. This argument follows the standard construction for an operator system (see \cite{hamana1979injective}). One begins by showing that there exists a minimal $\mathcal V$-seminorm. This is proven by taking a net of $\mathcal V$-maps such that the induced $\mathcal V$-seminorms form a descending chain and showing such chain has a lower bound. Appealing to Zorn's lemma, the existence is established. The next step is showing if $\vp: \mathcal D^I \to \mathcal D^I$ is a $\mathcal V$-projection such that the induced seminorm $p_\vp$ is minimal then $\vp$ is a minimal $\mathcal V$-projection. One then shows $\vp(\mathcal D^I)$ is an injective envelope for the $k$-AOU space $\mathcal V.$ In order to prove injectivity, one first shows a $k$-AOU space is injective if and only if there exists a $k$-positive projection $P: \mathcal D^I \to \mathcal V$ for some set $I$.  
\end{remark}

We conclude with an observation concerning the C*-envelope of a $k$-minimal operator system. This observation will be crucial for our subsequent results. By a result of Choi and Effros \cite[Theorem 3.1]{choi1977injectivity}, every injective operator system is completely order isomorphic to a C*-algebra. In particular, $I(\mathcal{V})$ is completely order isomorphic to a C*-algebra for any operator system $\mathcal{V}$. It was shown by Hamana that the image of $\mathcal{V}$ in this C*-algebra generates the \textit{C*-envelope} $C^*_e(\mathcal{V})$. Hence, $C^*_e(\mathcal{V})$ is completely order isomorphic to a subsystem of $I(\mathcal{V})$.

\begin{corollary} \label{cor: C^*_e(V) k-minimal when V is k-minimal}
Suppose that $\mathcal{V}$ is a $k$-minimal operator system. Then $C^*_e(\mathcal{V})$ is a $k$-minimal operator system.
\end{corollary}

\begin{proof}
By Proposition \ref{prop: k-minimal system if and only if subsystem of ell infinity}, there exists a unital complete order embedding $i: I(\mathcal{V}) \to \mathcal{D}^I$. Thus, there exists a unital complete order embedding $j: C^*_e(\mathcal{V}) \to \mathcal{D}^I$. By Proposition \ref{prop: k-minimal system if and only if subsystem of ell infinity} again, $C^*_e(\mathcal{V})$ is a $k$-minimal operator system. \qedhere
\end{proof}

\section{$k$-minimal C*-algebras}\label{sec: k minimal Cstar algebras}

\begin{definition}\label{defn: k minimal C star algebra}
Let $k \in \mathbb N$ and let $\mathcal A$ be a unital C*-algebra. We say $\mathcal A$ is \emph{$k$-minimal} if it is $k$-minimal as an operator system. 
\end{definition}

In the previous section, we observed that when $\mathcal{V}$ is a $k$-minimal C*-algebra, each of $\mathcal{V}$ and $C^*_e(\mathcal{V})$ admit a unital complete order embedding into a direct sum of the matrix algebra $M_k$. In this section, we will show that $k$-minimal C*-algebras can be faithfully embedded into a direct sum of matrix algebras, each with dimension no more than $k^2$. The result implies that whenever $\mathcal{V}$ is a $k$-minimal operator system, $C^*_e(\mathcal{V})$ embeds algebraically into a direct sum of matrix algebras. This will have important consequences for quantum correlations in the next section.

\begin{definition}
Given an Archimedean order unit space $(\mathcal V, C, e)$, a state $\vp: \mathcal V \to \mathbb C$ is called \emph{pure} if it is an extreme point in the set of states of $\mathcal V$. 
\end{definition}

The next lemma is likely known. We provide a brief proof for the sake of completeness.

\begin{lemma} \label{lem: Hahn-Banach pure states}
Let $\mathcal V \subseteq \mathcal W$ be an inclusion of AOU spaces and let $\vp: \mathcal V \to \mathbb{C}$ be a pure state. Then there exists a pure extension $\tilde{\varphi}: \mathcal W \to \mathbb{C}$.
\end{lemma}

\begin{proof}
Let $E$ denote the set of all extensions of $\vp$ to $\mathcal W$. Then $E$ is a non-empty convex subset of the state space of $\mathcal W$. Moreover, $E$ is closed in the weak-$*$ topology, since any weak-$*$ limit of functionals in $E$ must converge to an extension of $\vp$. By the Krein-Milman theorem, $E$ is the convex hull of its extreme points. It remains to check that an extreme point of $E$ is a pure state on $\mathcal W$. Indeed, suppose that $\psi \in E$ and $\psi = t \psi_1 + (1-t) \psi_2$ where $\psi_1$ and $\psi_2$ are states on $\mathcal W$. Since $\varphi$ is pure and since $\varphi = t \psi_1 |_{\mathcal V} + (1-t) \psi_2|_{\mathcal V}$, it must be that $\psi_1|_{\mathcal V} = \psi_2|_{\mathcal V} = \varphi$. So $E$ is a face of the state space of $\mathcal W$. In particular, the extreme points of $E$ are pure states for $\mathcal W$. 
\end{proof}

In the following, we let $C(X,\mathcal{A})$ denote the C*-algebra of continuous functions from a compact Hausdorff space $X$ to a unital C*-algebra $\mathcal{A}$ with the multiplication, norm, and adjoint defined pointwise on $X$.

\begin{lemma} \label{lem: pure states on C(X,A)}
Let $X$ be a compact Hausdorff space and let $\mathcal{A}$ be a unital C*-algebra. Suppose that $\rho: C(X, \mathcal{A}) \to \mathbb{C}$ is a pure state. Then there exists a pure state $\vp: \mathcal{A} \to \mathbb{C}$ and a point $x_0 \in X$ such that for all $f \in C(X, \mathcal{A})$, $\rho(f) = \vp(f(x_0))$.
\end{lemma}

\begin{proof}
Let $S$ denote the state space of $\mathcal A$, which is necessarily a weak-$*$ compact Hausdorff space. Let $C(X \times S)$ denote the abelian C*-algebra of continuous functions on the Cartesian product of the sets $X$ and $S$ with the product topology. Observe that the mapping $i: C(X, \mathcal{A}) \to C(X \times S)$ given by
\[ i(f)(x,\varphi) = \varphi(f(x)) \]
is a unital order embedding of AOU spaces. Identifying $C(X,\mathcal{A})$ with its image under $i$, by Lemma~\ref{lem: Hahn-Banach pure states} we may extend $\rho$ to a pure state, $\tilde{\rho}$, on $C(X \times \mathcal{S})$. Since $\tilde{\rho}$ is pure, there exists a point $(x_0, \vp) \in X \times \mathcal{S}$ such that $\tilde{\rho}(g)=g(x_0,\vp)$ for all $g \in C(X \times \mathcal{S})$. Consequently, $\rho(f)=\vp(f(x_0))$ for all $f \in C(X,\mathcal{A})$. Finally, since $\rho$ is pure, $\vp$ must be pure. Otherwise, if $\vp = t \vp_1 + (1-t)\vp_2$, then we could write $\rho = t\rho_1 + (1-t)\rho_2$ where $\rho_i(f) = \vp_i(f(x_0))$ for each $i \in \{1,2\}$.
\end{proof}

\begin{lemma} \label{lem: pure states k-minimal}
Let $k \in \mathbb{N}$ and suppose $\mathcal{A}$ is a $k$-minimal C*-algebra. Let $\rho: \mathcal{A} \to \mathbb{C}$ be a pure state. If $\pi_{\rho}: \mathcal{A} \to B(H)$ is the corresponding GNS representation, then $\dim(H) \leq k$.
\end{lemma}

\begin{proof}
Let $X$ denote the set of all $k$-states of $\mathcal A$; i.e., linear maps $\psi: \mathcal{A} \to M_k$ that are unital completely positive. Then $X$ is a weak-$*$ compact Hausdorff space. Since $\mathcal{A}$ is $k$-minimal, the mapping $i: \mathcal{A} \to C(X,M_k)$ given by $a \mapsto \hat{a}$ is a unital complete order embedding, where $\hat{a}(\vp) := \vp(a)$ for every $\vp \in X$. Note since each $\vp \in X$ is unital it follows $\hat{1}_{\mathcal A}(\vp) = \vp(1_{\mathcal A}) = I_k$. Injectivity of $i$ will follow from $i^{-1}: i(\mathcal A) \to \mathcal A$ being completely positive (see below), along with the fact that the cones of $\mathcal A$ are proper.
Thus, the mapping $i: \mathcal{A} \to C(X,M_k)$ is a unital linear embedding. Since $\mathcal A$ is $k$-minimal, Proposition \ref{prop: k-min and OMIN_k} implies $a \in M_n(\mathcal A)^+$ if and only if $\vp_n(a) \in M_{kn}^+$ for every unital $k$-positive (and thus completely positive) $\vp: \mathcal A \to M_k$. It is then immediate that both $i: \mathcal{A} \to C(X,M_k)$ and $i^{-1}: i(\mathcal A) \to \mathcal A$ are completely positive. This proves the map $i$ is a unital complete order embedding.

Let $\rho$ be a pure state on $\mathcal{A}$. Then $\rho \circ i^{-1}: i(\mathcal A) \to \mathbb C$ is a pure state. By Lemma \ref{lem: Hahn-Banach pure states}, $\rho \circ i^{-1}$ extends to a pure state $\tilde{\rho}:C(X,M_k) \to \mathbb C$. By Lemma \ref{lem: pure states on C(X,A)}, there exists a point $\varphi \in X$ and a vector $h \in \mathbb{C}^k$ such that for all $f \in C(X,M_k)$ we have $\tilde{\rho}(f) = \langle f(\varphi)h, h \rangle$, since the pure states on $M_k$ are vector states. It follows that $\rho(a) = \langle \varphi(a)h,h \rangle$ for some $\varphi \in X$. Hence, the Hilbert space $H$ in the GNS representation $\pi_{\rho}: \mathcal{A} \to B(H)$ is unitarily equivalent to a quotient of the Hilbert space $i(\mathcal{A})h \subseteq \mathbb{C}^k$ equipped with the inner product $\langle i(a)h, i(b)h \rangle_{i(\mathcal{A})h} := \langle i(b^*a)h, h \rangle$.  Thus  $\dim(H) \leq k$.
\end{proof}

We are now able to prove that if a C*-algebra is $k$-minimal (as an operator system), then it embeds into an $\ell_\infty$-direct sum of matrix algebras of size less than or equal to $k$. 

\begin{theorem}\label{thm: k-minimal C*algebra is a subalgebra of ell infinity of matrix algebras}
Let $k \in \mathbb{N}$ and suppose $\mathcal{A}$ is a $k$-minimal C*-algebra. Then there exists a set $I$ and a faithful unital $*$-homomorphism $\pi: \mathcal{A} \to \bigoplus_{i \in I} M_{d_i}$ where $d_i \leq k$ for each $i \in I$.
\end{theorem}

\begin{proof}
This follows from Lemma~\ref{lem: pure states k-minimal} and the observation that the direct sum of the GNS representations \[ \bigoplus_{ \rho \in \mathcal{P}} \pi_{\rho}: \mathcal A \to B(H), \quad H:= \bigoplus_{\rho \in \mathcal P}H_\rho, \quad \dim H_\rho \leq k, \] 
where $\mathcal{P}$ is the set of pure states on $\mathcal{A}$, 
is a faithful unital $*$-homomorphism.
\end{proof}

\begin{remark}
Theorem \ref{thm: k-minimal C*algebra is a subalgebra of ell infinity of matrix algebras}, together with Corollary \ref{cor: C^*_e(V) k-minimal when V is k-minimal}, imply that every $k$-AOU space $\mathcal{V}$ has a natural C*-envelope in the category $k$-AOU, which we identify with the $k$-AOU space $C^*_e(V_\text{$k$-min})$. By natural, we mean that the C*-envelope of the $k$-AOU space $\mathcal{V}$ is not only a subspace of a direct sum of matrix algebras $M_{d_i}$ with $d_i \leq k$, but is in fact a C*-subalgebra of such a direct sum, up to $*$-isomorphism.
\end{remark}

\section{Projections in $k$-AOU spaces}\label{sec: projections in k AOU spaces}

In this section, we will characterize when a positive contraction $p$ in a $k$-AOU space $\mathcal{V}$ is a projection. In the language of \cite{araiza2020abstract}, we seek to determine when $p$ is an abstract projection in the operator system $\mathcal{V}_\text{$k$-min}$ using only the data of the $k$-AOU space, namely the triple $(\mathcal{V}, C, e)$. We begin by reviewing the notion of an abstract projection in an operator system from \cite{araiza2020abstract}.

Suppose that $(\mathcal{V}, \mathbfcal{C}, e)$ is an operator system, and suppose that $p \in \mathcal{V}$ is a positive contraction. It follows that $p^\perp = e - p$ is also a positive contraction. For each $n \in \mathbb{N}$, we define a set
$\mathbfcal{C}(p)_n$ to be the set of $x \in M_{2n}(\mathcal{V})$ with the property that $x=x^*$ and for every $\epsilon > 0$, there exists $t > 0$ such that 
\[ x  + \epsilon I_n \otimes p \oplus p^\perp + t I_n \otimes p^\perp \oplus p \in \mathbfcal{C}_{2n}. \]
In general, each set $\mathbfcal{C}(p)_n$ is a cone, but not a proper cone. If we set $\mathcal{J} = \Span \{ \mathbfcal{C}(p)_1 \cap - \mathbfcal{C}(p)_1\}$, then the quotient $M_2(\mathcal{V}) / \mathcal{J}$ is an operator system when equipped with the matrix ordering $\{\mathbfcal{C}(p)_n + M_n(\mathcal{J}) \}_{n=1}^\infty$ and the order unit $e \otimes J_2 + \mathcal{J}$, where
\[ J_2 = \begin{pmatrix} 1 & 1 \\ 1 & 1 \end{pmatrix}. \] 
(See Theorem 4.9 of \cite{araiza2020abstract} for details.) Moreover, the mapping $\pi_p: \mathcal{V} \to M_2(\mathcal{V})/\mathcal{J}$ defined by
\[ \pi_p(x) = x \otimes J_2 + \mathcal{J} \]
is unital and completely positive (Proposition 3.2 of \cite{araiza2021universal}). The following theorem shows that $\pi_p$ can be used to characterize when $p$ is a projection in $\mathcal{V}$.

\begin{theorem}[Theorems 5.3, 5.7, 5.8 of \cite{araiza2020abstract}] \label{thm: abstract projections in operator systems}
Let $\mathcal{V}$ be an operator system, and suppose that $p \in \mathcal{V}$ is a positive contraction. Then the following statements are equivalent:
\begin{enumerate}
    \item There exists a Hilbert space $H$ and a complete order embedding $\varphi: \mathcal{V} \to B(H)$ such that $\varphi(p)$ is a projection on $H$.
    \item $p$ is a projection in $C^*_e(\mathcal{V})$.
    \item $\pi_p$ is a complete order embedding.
\end{enumerate}
\end{theorem}

\begin{example}
If $\mathcal{V}$ is an operator system and $p \in \mathcal{V}$ is an element, it is possible for one complete order embedding of $\mathcal{V}$ into $B(H)$ to take $p$ to a projection while another complete order embedding of $\mathcal{V}$ into $B(H)$ will not.

For instance, let $\mathcal{V} := \left\{ \left( \begin{smallmatrix} a & 0 \\ 0 & b \end{smallmatrix} \right) : a,b \in \mathbb{C} \right\}$ be the operator system of diagonal $2 \times 2$ matrices, and let $p := \left( \begin{smallmatrix} 1 & 0 \\ 0 & 0 \end{smallmatrix} \right)$.  Consider the canonical complete embedding $i : \mathcal{V} \to B(\C^2) \cong M_2(\C)$ with $i(A) = A$, which takes a $2 \times 2$ matrix to the operator given by left multiplication by that matrix. We see that $i(p) = \left( \begin{smallmatrix} 1 & 0 \\ 0 & 0 \end{smallmatrix} \right)$ is a projection in $B(\C^2)$.

On the other hand, consider the state $s : \mathcal{V} \to \C$ given by $s\left( \begin{smallmatrix} a & 0 \\ 0 & b \end{smallmatrix} \right) = \frac{a+b}{2}$.  Define $T : \mathcal{V} \to B(\mathbb{C}^2 \oplus \mathbb{C}) \cong M_3(\mathbb{C})$ by $T = i \oplus s$.  (In particular, with our identification $T \left( \begin{smallmatrix} a & 0 \\ 0 & b \end{smallmatrix} \right) = a \oplus b \oplus \frac{a+b}{2} = \left( \begin{smallmatrix} a & 0 & 0 \\ 0 & b & 0 \\ 0 & 0 & (a+b)/2 \end{smallmatrix} \right)$.)  Since $i$ is completely positive and states are completely positive \cite[Proposition~3.8]{paulsen2002completely}, $T = i \oplus s$ is completely positive.  Moreover, the injectivity of $i$ implies $T$ is injective.  Finally, since $i$ is a complete order embedding, $T(A) \geq 0$ implies $i(A) \geq 0$ implies $A \geq 0$, and hence $T$  is a complete order embedding.  Thus $T$ is a complete order embedding, but $T(p) = 1 \oplus 0 \oplus \frac{1}{2}$ is not a projection.
\end{example}

\begin{definition}
Let $\mathcal V$ be an operator system and let $p \in \mathcal V$ be a positive contraction. We call $p$ a \emph{$k$-order abstract projection} if the mapping $\pi_p: \mathcal V \to M_2(\mathcal V)/ \mathcal J, x \mapsto x \otimes J_2 + \mathcal J,$ is a $k$-order embedding. 
\end{definition}

Consider a $k$-minimal operator system $\mathcal V$ and let $p \in \mathcal V$ be a $k$-order abstract projection. By \cite[Proposition 3.2]{araiza2021universal} we have $\pi_p: \mathcal V \to M_2(\mathcal V)/ \mathcal J$ is completely positive. Consider the map $\pi_p^{-1}: \pi_p(\mathcal V) \to \mathcal V$. Since $\mathcal V = \mathcal V_\text{$k$-min}$ the map $\pi_p^{-1}$ is completely positive if and only if it is $k$-positive (see \cite[Theorem 3.8]{xhabli2012super}). Since $p$ is a $k$-order abstract projection then $\pi_p^{-1}: \pi_p(\mathcal V) \to \mathcal V$ is $k$-positive and therefore $\pi_p^{-1}: \pi_p(\mathcal V) \to \mathcal V$ is completely positive by Proposition \ref{prop: k-positive implies cp on k-min}. Since $\pi_p: \mathcal V \to \pi_p(\mathcal V)$ is a linear isomorphism, by assumption we have $\pi_p$ is a complete order embedding. Consequently, $p$ is an abstract projection. We record these remarks here: 

\begin{proposition}\label{prop: k-order projection is projection in k-minimal operator system}
Let $\mathcal V$ be a $k$-minimal operator system and let $p \in \mathcal V$ be a positive contraction. Then $p$ is an abstract projection if and only if $p$ is a $k$-order abstract projection.
\end{proposition}

\begin{definition} \label{defn: C[p]}
Let $(\mathcal{V}, C, e)$ be a $k$-AOU space and assume $I_k \otimes p \in M_k(\mathcal{V})$ is a positive contraction. Define the set $C[p]$ to be the set of $x \in M_k(\mathcal{V})$ such that $x^*=x$ and for every $\epsilon > 0$ there exists $t > 0$ such that 
\begin{equation} (\alpha^* + \beta^*)x(\alpha + \beta) + (\epsilon \alpha^* \alpha + t \beta^* \beta) \otimes p + (\epsilon \beta^* \beta + t \alpha^* \alpha) \otimes p^{\perp} \in C \label{eq: defn of C[p]}
\end{equation}
for every $\alpha, \beta \in M_k$.
\end{definition}

The relevance of Definition \ref{defn: C[p]} is illustrated by the following proposition.

\begin{proposition}\label{prop: k-min system then p k-order projection iff coincidence of C_k and C_k(p)} 
Let $(\mathcal{V},\mathbfcal{C},e)$ be a $k$-minimal operator system. Then a positive contraction $p \in \mathcal{V}$ is a $k$-order abstract projection if and only if $\mathbfcal C_k = \mathbfcal C_k[p]$.
\end{proposition}

\begin{proof}
Suppose $\mathbfcal C_k = \mathbfcal C_k[p]$, and suppose that $x \otimes J_2 \in \mathbfcal{C}(p)_k$. Thus $x = x^*$ and for all $\epsilon >0$ there exists $t>0$ such that  \begin{align*}
    x \otimes J_2 + \epsilon I_k \otimes (p \oplus p^{\perp}) + t I_k \otimes (p^{\perp} \oplus p) \in \mathbfcal C_{2k}.
\end{align*} Applying the canonical shuffle $M_k(\mathcal V) \otimes M_2 \to M_2 \otimes M_k(\mathcal V)$, and using that $\mathbfcal C$ is a matrix ordering, then  conjugation by \[ T = \begin{pmatrix}
\alpha \\
\beta
\end{pmatrix} \] implies $x \in \mathbfcal C_k[p]$ and thus $x \in \mathbfcal C_k.$ Therefore $\pi_p$ is a $k$-order embedding, making $p$ a $k$-order abstract projection in $\mathcal{V}$. 

Conversely, suppose $p$ is a $k$-order abstract projection in $\mathcal{V}$. Thus $\pi_p$ is a $k$-order embedding. If $x \in \mathbfcal C_k$ then for every $\epsilon >0$ and every $t>0$ we have  \begin{align*}
    x \otimes J_2 + \epsilon I_k \otimes (p \oplus p^{\perp}) + t I_k \otimes (p^{\perp} \oplus p) \in \mathbfcal C_{2k}.
\end{align*} Thus, given any $\alpha, \beta \in M_k$, Equation~(\ref{eq: defn of C[p]}) is satisfied after applying a canonical shuffle and conjugating by $T$. Thus, $x \in \mathbfcal C_k[p]$, and we have established $\mathbfcal C_k \subseteq \mathbfcal C_k[p]$.  Now suppose $x \in \mathbfcal C_k[p]$ and let $\epsilon >0$ be arbitrary. Then there exists $t>0$ such that \[
\begin{pmatrix}
\alpha \\
\beta
\end{pmatrix}^*(J_2 \otimes x + \epsilon (p \oplus p^\perp) \otimes I_k + t (p^\perp \oplus p) \otimes I_k)\begin{pmatrix}
\alpha \\
\beta
\end{pmatrix} \in \mathbfcal C_k
\] for all $\alpha,\beta \in M_k$. After applying a canonical shuffle we have $x \otimes J_2 + \epsilon I_k \otimes (p \oplus p^{\perp}) + t I_k \otimes (p^{\perp} \oplus p) \in \mathbfcal C_{2k}^\text{$k$-min}$, which implies $x \otimes J_2 + \epsilon I_k \otimes (p \oplus p^{\perp}) + t I_k \otimes (p^{\perp} \oplus p) \in \mathbfcal C_{2k}$ by our assumption that $\mathcal V$ is $k$-minimal. Since $\pi_p$ is a $k$-order embedding, $x \in \mathbfcal C_k$. Therefore $\mathbfcal C_k = \mathbfcal C_k[p]$ \qedhere
\end{proof}

\begin{remark}
If $k \in \mathbb N$ and if $(\mathcal V, C, e)$ is a $k$-AOU, space then by Corollary~\ref{cor: extensions of k-AOU spaces agree on levels preceeding k} we know all possible extensions of the $k$-AOU space yield the same order structure up to the $k$\textsuperscript{th} level. This is to say that given an extension $\mathbfcal C$ of $(\mathcal V, C, e)$, then the cones $\mathbfcal C_1,\dotsc,\mathbfcal C_k$, are unique. Thus, there is no ambiguity in saying an element $p \in \mathcal V$  is a positive contraction, since this means $p$ is positive in some (and hence every) extension of $(\mathcal V, C, e).$ Moreover, by Corollary~\ref{cor: positive at 1st level} we have for $p \in \mathcal V$ that $I_k \otimes p \in C$ if and only if $p \in \mathbfcal C_1$ for every (equivalently, any) extension $\mathbfcal C$, of $(\mathcal V, C, e).$ These observations imply the following corollary.
\end{remark}

\begin{corollary}\label{cor: p k-order abstract projection relative to all extensions}
Let $k \in \mathbb{N}$ and let $(\mathcal V, C,e)$ be a $k$-AOU space. If $p \in \mathcal V$ such that $I_k \otimes p$ is a positive contraction and $C[p] = C$, then given any extension $\mathbfcal C$ of $C$ it follows $\mathbfcal C_k[p] = \mathbfcal C_k.$ 
\end{corollary}

In light of these results, we call a positive contraction $p$ in a $k$-AOU space $(\mathcal{V},C,e)$ an \textit{abstract projection} if $C[p]=C$. Combining the results above, we have proven the following.

\begin{theorem} \label{thm: abstract projections in k-AOU}
Let $k \in \mathbb{N}$ and let $(\mathcal{V},C,e)$ be a $k$-AOU space. Suppose that $p \in \mathcal{V}$ is a positive contraction. Then the following statements are equivalent:
\begin{enumerate}
    \item For every extension $\mathbfcal{C}$ of $C$, $\mathbfcal{C}_k[p] = \mathbfcal{C}_k$.
    \item $p$ is an abstract projection in $\mathcal{V}_\text{$k$-min}$.
    \item $p$ is an abstract projection in the $k$-AOU space $(\mathcal{V},C,e)$.
\end{enumerate}
\end{theorem}

Given a $k$-AOU space $\mathcal{V}$, let $\text{Proj}(\mathcal{V})$ denote the set of all abstract projections in $\mathcal{V}$. Using a result from \cite{araiza2020abstract}, together with results from the previous section, it turns out that every element of $\text{Proj}(\mathcal{V})$ is a projection in the C*-algebra $C^*_e(\mathcal{V}_\text{$k$-min})$. 

\begin{corollary} \label{cor: projection in C*-envelope}
Let $\mathcal{V}$ be a $k$-AOU space. Then $p$ is an abstract projection in $\mathcal{V}$ if and only if $p$ is a projection in $C^*_e(\mathcal{V}_\text{$k$-min})$. Consequently, $p$ is an abstract projection in $\mathcal{V}$ if and only if there exists a $k$-order embedding $\pi: \mathcal{V} \to \bigoplus_{i \in I, d_i \leq k} M_{d_i}$ such that $\pi(p)$ is projection.
\end{corollary}

\begin{proof}
If $p$ is a projection in $C^*_e(\mathcal{V}_\text{$k$-min})$, then $p$ is an abstract projection in $\mathcal{V}$ by Theorem \ref{thm: abstract projections in k-AOU}. On the other hand, if $p$ is an abstract projection in $\mathcal{V}$, then Theorem \ref{thm: abstract projections in k-AOU} implies that $p$ is an abstract projection in $\mathcal{V}_\text{$k$-min}$. By \cite[Theorem 5.8]{araiza2020abstract}, $p$ is a projection in $C^*_e(\mathcal{V}_\text{$k$-min})$. The final statement follows from Theorem \ref{thm: k-minimal C*algebra is a subalgebra of ell infinity of matrix algebras}, which implies there exists a set $I$ and a faithful $*$-homomorphism $\pi: C^*_e(\mathcal{V}_\text{$k$-min}) \to \bigoplus_{i \in I, d_i \leq k} M_{d_i}$.
\end{proof}

In the next section, we will use the above corollary to obtain an application to the theory of quantum correlations.

\section{Quantum correlations from $k$-AOU spaces} \label{sec: quantum correlations and k-AOU spaces}

We begin this section by reviewing definitions and various notions from the theory of correlation sets. Following the standard notation, given a natural number $n \in \mathbb N$ we denote the set $[n]:= \{1,\dotsc,n\}$.  Given two natural numbers $m,n \in \mathbb N$, a \emph{correlation with n inputs and m outputs} is defined to be a tuple $p:= \{p(ab|xy)\}_{a,b \in [m], x,y \in [n]} \subseteq \mathbb R^{m^2 n^2},$ such that for each $x,y \in [n]$ it follows $\sum_{ab} p(ab|xy) = 1$ and $p(ab|xy) \geq 0$ for each $x,y \in [n]$ and $a,b \in [m].$ The elements $x,y$ are known as the \emph{inputs} of the correlation and the elements $a,b$ are known as the \emph{outputs}. The set of all correlations with $n$ inputs and $m$ outputs will be denoted $C(n,m)$, and such a correlation will be denoted simply as $p$ when no confusion will arise. Correlations represent bipartite systems where two parties, typically known as Alice and Bob, receive an input according to a particular probability distribution, and make measurements and produce two outputs, $a$ and $b$ respectively. Thus, given inputs $x,y \in [n]$, the positive real number $p(ab|xy)$ represents the probability that Alice produces output $a$ and Bob produces output $b$ given that Alice receives input $x$ and Bob receives input $y$. If $p \in C(n,m)$ is a correlation then we say it is \emph{nonsignalling} if the values \begin{align*}
    p_A(a|x):= \sum_b p(ab|xy), \quad \text{and} \quad p_B(b|y):= \sum_a p(ab|xy),
\end{align*} are well-defined for each $a,b,x,y$. This is to say the value $p_A(a|x)$ is independent of the input $y \in [n]$, and the value $p_B(b|y)$ is independent of the input $x \in [n]$. These values are known as the \emph{marginal densities} relative to the correlation $p$. The set of all nonsignalling correlations with $n$ inputs and $m$ outputs will be denoted by $C_{ns}(n,m).$ It is well known that the set of nonsignalling correlations is a convex polytope of $\mathbb R^{m^2 n^2}$. Nonsignalling correlations tell us that the two parties Alice and Bob both perform their measurements on their respective inputs independently without communication with one another. Of particular interest are various subsets of nonsignalling correlations. The study of correlations that are determined by suitable projections, with other natural properties, have generated tremendous interest in both the fields of operator algebras and quantum information theory. We now consider such subsets. 

Let $H$ be a Hilbert space. Then a \emph{projection-valued measure (PVM)} is a tuple of projections $\{E_i\}_{i=1}^m \subseteq B(H)$ such that $\sum_i E_i = \id_H$. Note this latter property necessarily implies the projections are pairwise orthogonal. A correlation $p \in C_{ns}(n,m)$ will be called a \emph{quantum commuting} correlation if there exists a Hilbert space $H$, a unit vector $\eta \in H$, and projection-valued measures $\{E_{xa}\}_{a \in [m]}, \{F_{yb}\}_{b \in [m]}$ for each input $x,y \in [n]$, such that $E_{xa}F_{yb} = F_{yb}E_{xa}$ for each $a,b,x,y$ and such that $p(ab|xy) = \langle \eta | E_{xa}F_{yb}\eta \rangle_H.$ The set of all quantum commuting correlations with $n$ inputs and $m$ outputs will be denoted $C_{qc}(n,m)$. If we require that the Hilbert space $H$ be finite-dimensional, then the correlation $p$ will be called a \emph{quantum} correlation. The set of all quantum correlations with $n$ inputs and $m$ outputs will be denoted $C_q(n,m)$. The closure $\overline{C_q(n,m)}$ will be denoted $C_{qa}(n,m)$ and we will call this the set of \emph{quantum approximate} correlations with $n$ inputs and $m$ outputs. For all input-output pairs $(n,m)$, we have the following string of inclusions \begin{align*}
    C_q(n,m) \subseteq C_{qa}(n,m) \subseteq C_{qc}(n,m) \subseteq C_{ns}(n,m).
\end{align*} For more details on quantum and quantum commuting correlations, we refer the interested reader to \cite[Section 2]{Paulsenchromatic}. 

Using the tools from the previous sections, as well as previous work of the first two authors \cite{araiza2020abstract}, our goal is to characterize quantum correlations via $k$-order abstract projections and $k$-AOU spaces. 

\begin{definition}\label{defn: quantum k-operator system}
A $k$-AOU space $(\mathcal V, C, e)$ will be called a \emph{quantum $k$-AOU space} if $\mathcal V = \Span \{Q(ab|xy): a,b \in [m], x,y \in [n]\}$ such that for each $x,y \in [n]$ one has $\sum_{ab}Q(ab|xy) = e$, the operators \begin{align*}
    E(a|x):= \sum_b Q(ab|xy),\quad \text{and} \quad F(b|y):= \sum_a Q(ab|xy),
\end{align*} are well-defined for each $x,y \in [n]$ and $a,b \in [m],$ and each generator $Q(ab|xy)$ is a an abstract projection in $\mathcal{V}$.
\end{definition}

\begin{remark}
In \cite{araiza2020abstract} the notion of a \emph{quantum commuting} operator system was introduced. In Definition~\ref{defn: quantum k-operator system}, if we instead require $\mathcal V$ to be an operator system, and if the generators $Q(ab|xy)$ are all abstract projections (as defined in \cite[Definition 5.4]{araiza2020abstract}, see also \cite[Theorem 3.3, Definition 4.4]{araiza2021universal}), then we say the operator system $\mathcal V$ is quantum commuting. By Theorem~\ref{thm: abstract projections in k-AOU} it follows whenever $\mathcal{V}$ is a quantum $k$-AOU space, then $\mathcal{V}_\text{$k$-min}$ is a quantum commuting operator system.
\end{remark}

We will make use of the following fact (see e.g. \cite[Proposition 6.1]{araiza2020abstract}): 
\begin{proposition} \label{prop: equivalent to correlation}
Let $n,m \in \mathbb N$ and let $p \in C(n,m)$. Then the following statements are equivalent.
\begin{enumerate}
    \item $p \in C_\text{qc}(n,m)$ (resp. $p \in C_\text{q}(n,m)).$
    \item There exists a (resp. finite dimensional) C*-algebra $\mathcal A$, projection valued measures $\{E_{xa}\}_{a=1}^m, \{F_{yb}\}_{b=1}^m \subseteq \mathcal A$ for each $x,y \in [n]$ satisfying $E_{xa}F_{yb} = F_{yb}E_{xa}$ for all $x,y \in [n]$ and $a,b \in [m]$, and a state $\varphi: \mathcal A \to \mathbb{C}$ such that $p(ab|xy) = \varphi(E_{xa}F_{yb})$.
    \item There exists an operator system $\mathcal V \subseteq B(H)$ (resp. for a finite dimensional Hilbert space $H$), projection valued measures $\{E_{xa}\}_{a=1}^m, \{F_{yb}\}_{b=1}^m$ for each $x,y \in [n]$ satisfying $E_{xa}F_{yb} \in \mathcal V$ and $E_{xa}F_{yb} = F_{yb}E_{xa}$ for all $x,y \in [n]$ and $a,b \in [m]$, and a state $\varphi: \mathcal V \to \mathbb{C}$ such that $p(ab|xy) = \varphi(E_{xa}F_{yb})$.
\end{enumerate}
\end{proposition} We point out that quantum commuting correlations may also be described via states on operator system tensor products of particular $C^*$-algebras (see \cite[Corollary 3.2]{Lupiniperfect}).

\begin{theorem}\label{thm: quantum correlation characterization}
Fix $n,m \in \mathbb N$. Then a correlation $p \in C(n,m)$ is a quantum correlation if and only if there exists $k \in \mathbb N$ and a quantum $k$-AOU space $(\mathcal V, C,e)$ with generators $\{Q(ab|xy): a,b \in [m], x,y \in [n]\}$ and a state $\vp: \mathcal V \to \mathbb C$ such that $p(ab|xy) = \vp(Q(ab|xy))$ for each $a,b \in [m], x,y \in [n].$
\end{theorem}

\begin{proof}
Let $p \in C_q(n,m)$. Then by Proposition~\ref{prop: equivalent to correlation} there exists a finite-dimensional Hilbert space $H$, an operator system $\widetilde{\mathcal V} \subseteq B(H)$, and projection-valued measures $\{E_{xa}\}_{a=1}^m, \{F_{yb}\}_{b=1}^m \subseteq B(H)$, for each $x,y \in [n]$, such that $E_{xa}F_{yb} = F_{yb}E_{xa}$ and $E_{xa}F_{yb} \in \widetilde{\mathcal V}$ for each $x,y \in [n], a,b \in [m],$ and a state $\vp: \widetilde{\mathcal V} \to \mathbb C$ such that $p(ab|xy) = \vp(E_{xa}F_{yb})$. Let $k := \dim H$ and let $\mathcal V$ be the operator system spanned by the elements of the set $\{Q(ab|xy): a,b \in [m], x,y \in [n]\},$ where $Q(ab|xy):= E_{xa}F_{yb}$. Since for each $x,y \in [n]$ the projection-valued measures $\{E_{xa}\}_{a=1}^m$, $\{F_{yb}\}_{b=1}^m$ pairwise commute then $Q(ab|xy)$ is a projection for each $x,y,a,b.$ In particular, the set of generators $\{Q(ab|xy): a,b \in [m], x,y \in [n]\}$ consists of $k$-order abstract projections and it is straightforward to verify that the generators satisfy the conditions of Definition~\ref{defn: quantum k-operator system}. Since $\mathcal V$ is $k$-minimal, it follows from Theorem \ref{thm: abstract projections in k-AOU} that $\mathcal V$ is a quantum $k$-AOU space such that $\vp(Q(ab|xy)) = p(ab|xy)$ for each $a,b \in [m], x,y \in [n].$

Conversely, suppose $(\mathcal V, C,e)$ is a quantum $k$-AOU space with the set of generators $\{Q(ab|xy): a,b \in [m], x,y \in [n] \}.$ Each generator $Q(ab|xy) \in \mathcal V$ is an abstract projection in $\mathcal{V}$ and therefore a projection in $C^*_e(\mathcal{V}_\text{$k$-min})$ by Corollary~\ref{cor: projection in C*-envelope}. Theorem~\ref{thm: k-minimal C*algebra is a subalgebra of ell infinity of matrix algebras} implies $C_e^*(\mathcal V_\text{$k$-min})$ is a C*-subalgebra of $\bigoplus_{i \in I, d_i \leq k} M_{d_i},$ for some set $I$. Thus, each projection $Q(ab|xy)$ is a projection in $\bigoplus_{i \in I, d_i \leq k} M_{d_i}.$ 
 
We write $Q(ab|xy) = (Q_i(ab|xy))_{i \in I} \in \bigoplus_{i \in I, d_i \leq k} M_{d_i}$, where for each $i \in I, Q_i(ab|xy) \in M_{d_i}$ is the compression to the $i$\textsuperscript{th} block. Since $Q(ab|xy)$ is a projection then it follows each $Q_i(ab|xy)$ is a projection in $M_{d_i}$. Fix $i \in I$ and let $E_i(a|x):= \sum_b Q_i(ab|xy), F_i(b|y):= \sum_a Q_i(ab|xy)$ be the respective marginal operators. Since $Q(ab|xy)Q(a^\prime b^\prime|xy) = 0$ if either $a \neq a^\prime$ or $b \neq b^\prime$ then the same property holds for $Q_i(ab|xy)$. In particular, $\{E_i(a|x)\}_{a=1}^m$ and $\{F_i(b|y)\}_{b=1}^m$ are projection valued measures in $M_{d_i}$ and \[
E_i(a|x)F_i(b|y) = \sum_{a^\prime b^\prime}Q_i(ab^\prime|xy)Q_i(a^\prime b|xy) = Q_i(ab|xy).
\] Thus, given any state $\vp^i: M_{d_i} \to \mathbb C$, then if $p_i$ is the correlation defined as $p_i(ab|xy):= \vp(E_i(a|x)F_i(b|y))$ then $p_i \in C_q(n,m)$ by Proposition~\ref{prop: equivalent to correlation}. 

Let $\vp: \bigoplus_{i \in I, d_i \leq k} M_{d_i} \to \mathbb C$ be a state, and let $\vp^i:= \vp \vert_{M_{d_i}^\prime}: M_{d_i}^\prime \to \mathbb C$ be the restriction to the $i$\textsuperscript{th} block. Here we have let $M_{d_i}^\prime \subseteq \bigoplus_{i \in I, d_i \leq k} M_{d_i}$ be the closed subspace isomorphic to $M_{d_i}$. Thus, given $v \in M_{d_{i_o}}^\prime \subseteq \bigoplus_{i \in I, d_i \leq k} M_{d_i}$ then $v_i = 0$ for $i \neq i_o.$ Identifying each $Q_i(ab|xy)$ with the corresponding element $Q_i^\prime(ab|xy) \in M_{d_i}^\prime$ we have $Q(ab|xy) = \sum_i Q_i^\prime(ab|xy)$. It then follows \[
\vp(Q(ab|xy)) = \sum_i \lambda_i \vp^i(Q_i^\prime(ab|xy))
\] where $\sum_i \lambda_i = 1$. By our remarks above it follows if $p_i(ab|xy):= \vp^i(Q_i^\prime(ab|xy))$, then $p_i \in C_q(n,m)$. Thus, \[
\vp(Q(ab|xy)) = \sum_i \lambda_i p_i(ab|xy),
\] for each $x,y \in [n], a,b \in [m].$ We claim there exists a finite subset $I_o \subseteq I$ such that $\vp(Q(ab|xy)) = \sum_{i \in I_o} \lambda_i p_i(ab|xy)$. This follows from Carath\'{e}odory's theorem, since each $p_i \in C_q(n,m)$ and $C_q(n,m)$ is a convex subset of the finite dimensional vector space $\mathbb{R}^{m^2 n^2}$.   \qedhere
\end{proof}

\begin{remark}
In the above proof, we point out that our use of this generalized Carath\'{e}dory's theorem is due to Cook and Webster in \cite{CookCaratheodory}.
\end{remark}

By \cite{ji2020mip}, there exist positive integers $n,m$ such that $\overline{C_q(n,m)}$ is a proper subset of $C_{qc}(n,m)$. Consequently, the distance between the sets $C_{q}(n,m)$ and $C_{qc}(n,m)$ is at least some positive value $\epsilon > 0$, where distance is calculated in the $l_{\infty}$ norm on $\mathbb{R}^{m^2 n^2}$. (In fact, any norm on $\mathbb{R}^{m^2 n^2}$ would suffice.) Combining this observation with Theorem~\ref{thm: quantum correlation characterization} above and \cite[Theorem~6.3]{araiza2020abstract} yields the following statement, which may be interpreted as a new equivalent formulation of Tsirelson's conjecture in terms of operator systems and $k$-AOU spaces.

\begin{corollary} \label{cor: operator system Tsirleson}
Let $m,n \in \mathbb{N}$. Then $p \in C_{qc}(n,m) \setminus \overline{C_q(n,m)}$ if and only if there exists a quantum commuting operator system $\mathcal{V}$ with generators $E(ab|xy)$, a state $\varphi: \mathcal{V} \to \mathbb{C}$, and an $\epsilon > 0$ such that whenever $\mathcal{W}$ is a quantum $k$-AOU space with generators $F(ab|xy)$ and $\psi$ is a state on $\mathcal{W}$, then
\[ |\varphi(E(a'b'|x'y')) - \psi(F(a'b'|x'y'))| > \epsilon \]
for some $a',b' \in [m]$ and $x',y' \in [n]$.
\end{corollary}


\bibliographystyle{alpha}
\bibliography{References}

\end{document}